\documentclass[a4paper,12pt]{article}
    \usepackage[top=2cm,bottom=2cm,left=2.5cm,right=2.5cm]{geometry}
    \usepackage{cite, amsmath, amssymb}
    \usepackage[margin=1cm,%
                font=small,%
                format=hang,%
                labelsep=period,%
                labelfont=bf]{caption}
\usepackage[latin1]{inputenc}
\usepackage{amsmath}
\usepackage{amsfonts}
\usepackage{amssymb}
\usepackage{cite}
\usepackage{amsthm}
\usepackage{makeidx}
\usepackage{graphicx}
\usepackage{mathrsfs,xcolor}
\usepackage{enumerate}
\usepackage{blkarray}
\usepackage{bm}
\usepackage{setspace}
\usepackage{float}
\usepackage{authblk}
\usepackage{multirow}
\usepackage{booktabs}

\numberwithin{table}{section}
\numberwithin{equation}{section}

\theoremstyle{plain}
\newtheorem{theorem}{Theorem}[section]
\newtheorem{proposition}[theorem]{Proposition}
\newtheorem{definition}[theorem]{Definition}
\newtheorem{lemma}[theorem]{Lemma}
\newtheorem{example}[theorem]{Example}

\newtheorem{corollary}[theorem]{Corollary}
\newtheorem{remark}[theorem]{Remark}

\usepackage{authblk}
\author[1]{ \textbf{Bryan S. Hernandez}}
\author[2,3,4,5]{\textbf{Eduardo R. Mendoza}}

\affil[1]{\small \textit{Institute of Mathematics, University of the Philippines Diliman, Quezon City 1101, Philippines}}
\affil[2]{\small \textit{Mathematics and Statistics Department, De La Salle University, Manila  0922, Philippines}}
\affil[3]{\small \textit{Center for Natural Sciences and Environmental Research, De la Salle University, Manila 0922, Philippines}}
\affil[4]{\small \textit{Max Planck Institute of Biochemistry, Martinsried, Munich, Germany}}
\affil[5]{\small \textit{LMU Faculty of Physics, Geschwister -Scholl- Platz 1, 80539 Munich, Germany}}
\affil[*]{Email addresses: \texttt{bryan.hernandez@upd.edu.ph and mendoza@lmu.de}}

\title{\textbf{Positive Equilibria of\\Hill-Type Kinetic Systems}}

\date{\normalsize (September 2020)}

\begin{document}
\maketitle
\begin{abstract} 
	%This study explores determining whether a chemical reaction network $\mathscr{N}$ endowed with Hill-type kinetics $K$, called an HTK system $\left(\mathscr{N},K\right)$, has the capacity for multistationarity, i.e., whether there exist positive rate constants such that the corresponding differential equations admit at least two distinct positive equilibria within a stoichiometric class.
	This work introduces a novel approach to study properties of positive equilibria of a chemical reaction network $\mathscr{N}$ endowed with Hill-type kinetics $K$, called a Hill-type kinetic (HTK) system $\left(\mathscr{N},K\right)$, including their multiplicity and concentration robustness in a species.
	We associate a unique
	positive linear combination of power-law kinetic systems called
	poly-PL kinetic (PYK) system $\left( {\mathscr{N},{K_\text{PY}}} \right)$
	to the given HTK system.
	The associated system has the key property that its equilibria sets coincide with those of the Hill-type system, i.e.,
	${E_ + }\left( {\mathscr{N},K} \right) = {E_ + }\left( {\mathscr{N},{K_\text{PY}}} \right)$ and ${Z_ + }\left( {\mathscr{N},K} \right) = {Z_ + }\left( {\mathscr{N},{K_\text{PY}}} \right)$.
	This allows us to identify two novel subsets of the Hill-type kinetics, called PL-equilibrated and PL-complex balanced kinetics, to which recent results on absolute concentration robustness (ACR) of species and complex balancing at positive equilibria of power-law (PL) kinetic systems can be applied. Our main results also include the Shinar-Feinberg ACR Theorem for PL-equilibrated HT-RDK systems (i.e., subset of complex factorizable HTK systems), which establishes a foundation for the analysis of ACR in HTK systems,
	and the extension of the results of M\"uller and Regensburger on generalized mass action systems to PL-complex balanced HT-RDK systems. In addition, we derive the theory of balanced concentration robustness (BCR) in an analogous manner to ACR for PL-equilibrated systems.
	Finally, we provide further extensions of our results to a more general class of kinetics, which includes quotients of poly-PL functions.
	
\end{abstract}

\thispagestyle{empty}
\section{Introduction}
\label{sec:1}
This paper presents a novel approach to analyze properties of positive equilibria of a chemical reaction network $\mathscr{N}$ endowed with Hill-type kinetics $K$ called Hill-type kinetic (HTK) systems, including their multiplicity and concentration robustness in a species.
A Hill-type kinetics assigns to each $q$-th reaction, with $q = 1,..., r$, a function $K_q: \mathbb{R}_{ \ge 0}^\mathscr{S} \to \mathbb{R}$ of the form
\[{K_q}\left( x \right) = {k_q}\frac{{\prod\limits_{i = 1}^m {{x_i}^{{F_{qi}}}} }}{{\prod\limits_{i = 1}^m {\left( {{d_{qi}} + {x_i}^{{F_{qi}}}} \right)} }}\]
for $i=1,...,m$, where the rate constant $k_q>0$, $\mathscr{S}$ is the set of species, $F := (F_{qi})$ and $D :=  (d_{qi})$ are $r \times m$ real and nonnegative real matrices called the {\textit{kinetic order matrix}} and {\textit{dissociation constant matrix}}, respectively.
Furthermore, ${\rm supp} (D_q) = {\rm supp} (F_q)$ to ensure normalization of zero entries to 1.
In the following, we will denote the power-law numerator with $M_q(x)$ or $M_q$, and the ``translated power-law'' denominator with $T_q(x)$ or $T_q$.

Rate functions of this type were first studied by A. Hill in 1910 for the case of one species (i.e., $m=1$) and nonnegative integer exponents \cite{HILL1910}. Several years later, L. Michaelis and M. Menten focused their investigations in 1913 on functions with exponent 1 \cite{MIME1913}. Refinements and extensions of the Michaelis-Menten model were widely applied to enzyme kinetics \cite{SEGE1975} and also found their way into models of complex biochemical networks in Systems Biology.

The general class of such kinetics was first studied by Sorribas et al. in 2007 \cite{SHVA2007} as ``Saturable and Cooperative Formalism'' (SCF), a new modeling approach based on Taylor series approximation. While mass action and power-law kinetics required combinations of reactions to achieve saturable behavior, SCF achieved this with a single reaction, which is particularly useful in model reduction. SCF differs from Hill-type kinetics only in normalizing of the dissociation constants. A detailed comparison of SCF with the various biochemical modeling approaches can be found in the work of E. Vilaprinyo in 2007 \cite{VILA2007}.

The name ``Hill-type kinetics'' was introduced by C. Wiuf and E. Feliu in their 2013 paper \cite{WIFE2013}, which analyzed injectivity properties of these kinetics, also in relation to those of power-law kinetics and other classes with monotonic characteristics. Gabor et al. in 2015-2016 provided computational solutions to the linear conjugacy problem of systems called ?BioCRNs?, whose closely related kinetics are rational functions with mass action numerators and polynomials with nonnegative coefficients as denominators \cite{GAHA2015,GAHA2016}. A subset of complex factorizable Hill-type kinetics (denoted by HT-RDK) was determined in 2017 by Arceo et al. \cite{AJLM2017}. Finally, Nazareno et al. provided a computational solution to the linear conjugacy problem for any HTK system in 2019 \cite{NEML2019}.

Our approach is based on associating a unique poly-PL kinetic system $\left( {\mathscr{N},{K_\text{PY}}} \right)$ to any given HTK system $\left( {\mathscr{N},{K}} \right)$. The associated system has the key property that its equilibria sets coincide with those of the Hill-type system, i.e.,
${E_ + }\left( {\mathscr{N},K} \right) = {E_ + }\left( {\mathscr{N},{K_\text{PY}}} \right)$ and ${Z_ + }\left( {\mathscr{N},K} \right) = {Z_ + }\left( {\mathscr{N},{K_\text{PY}}} \right)$.
This relationship allows the identification of two novel subsets of Hill-type kinetics, called PL-equilibrated and PL-complex balanced kinetics, to which recent results on absolute concentration robustness (ACR) of species and complex balancing at positive equilibria of power-law (PL) kinetic systems can be applied.  Our main results include:
\begin{itemize}
	\item[$\bullet$] the Shinar-Feinberg ACR Theorem for PL-equilibrated HT-RDK systems as foundation for the analysis of ACR in Hill-type kinetic systems, and
	\item[$\bullet$] the extension of the results of M\"uller and Regensburger on generalized mass action system (GMAS) to PL-complex balanced HT-RDK systems.
\end{itemize}
Furthermore, we discuss the extension of results to poly-PL quotients (i.e., both the numerator and denominator are nonnegative linear combination of power-law functions) and more general kinetic systems.

The paper is organized as follows: Section \ref{prelim} collects the fundamental results on chemical reaction networks and kinetic systems needed in the later sections. After a brief review of poly-PL system basics, Section \ref{positive:equilibria:hill} introduces
the associated PYK system and its basic properties. A short comparison with a case-specific procedure for computing positive equilibria is made. In Section \ref{abs:con:rob:htk}, after concepts of absolute concentration robustness (ACR) in a species $X$ and recent results on PLK systems are reviewed, the Shinar-Feinberg ACR Theorem for the subset of PL-equilibrated HT-RDK systems is derived. It is applied to identify deficiency zero HT-RDK and minimally HT-NDK systems which also exhibit ACR in a species $X$. Using Feinberg's Theorem for independent network decompositions, larger and higher deficiency HTK systems with ACR in $X$ are identified. Section \ref{com:bal:eq:HTK} focuses on complex balanced equilibria of weakly reversible HT-RDK systems and first uses the associated PYK system to extend concepts and results of M\"uller and Regensburger on GMAS partially to them. For the subset of PL-complex balanced HT-RDK systems, the extended results are complete. For the same subset, the theory of balanced concentration robustness (BCR) is derived in an analogous manner to ACR for PL-equilibrated systems.
Section \ref{extension:beyond} discusses the extension to poly-PL quotients and more general kinetic sets.
Finally, summary, conclusions and outlook are provided in Section \ref{sec:sum}.

\section{Fundamentals of chemical reaction networks and kinetic systems}
\label{prelim}
\indent In this section, we present some fundamentals of chemical reaction networks and chemical kinetic systems. These are provided in \cite{arceo2015,feinberg,feinberg1}.

\subsection{Fundamentals of chemical reaction networks}

\begin{definition}
	A {\bf chemical reaction network} (CRN) $\mathscr{N} = \left(\mathscr{S},\mathscr{C},\mathscr{R}\right)$ of nonempty finite sets $\mathscr{S}$, $\mathscr{C} \subseteq \mathbb{R}_{\ge 0}^\mathscr{S}$, and $\mathscr{R} \subset \mathscr{C} \times \mathscr{C}$, of $m$ species, $n$ complexes, and $r$ reactions, respectively, such that
	\begin{itemize}
		\item[i.] $\left( {{C_i},{C_i}} \right) \notin \mathscr{R}$ for each $C_i \in \mathscr{C}$, and
		\item[ii.]  for each $C_i \in \mathscr{C}$, there exists $C_j \in \mathscr{C}$ such that $\left( {{C_i},{C_j}} \right) \in \mathscr{R}$ or $\left( {{C_j},{C_i}} \right) \in \mathscr{R}$.
	\end{itemize}
\end{definition}
We can view $\mathscr{C}$ as a subset of $\mathbb{R}^m_{\ge 0}$. The ordered pair $\left( {{C_i},{C_j}} \right)$ corresponds to the familiar notation ${C_i} \to {C_j}$.

\begin{definition}
	The {\bf molecularity matrix} $Y$ is an $m\times n$ matrix such that $Y_{ij}$ is the stoichiometric coefficient of species $X_i$ in complex $C_j$.
	The {\bf incidence matrix} $I_a$ is an $n\times r$ matrix such that 
	$${\left( {{I_a}} \right)_{ij}} = \left\{ \begin{array}{rl}
		- 1&{\rm{ if \ }}{C_i}{\rm{ \ is \ in \ the\ reactant \ complex \ of \ reaction \ }}{R_j},\\
		1&{\rm{  if \ }}{C_i}{\rm{ \ is \ in \ the\ product \ complex \ of \ reaction \ }}{R_j},\\
		0&{\rm{    otherwise}}.
	\end{array} \right.$$
	The {\bf stoichiometric matrix} $N$ is the $m\times r$ matrix given by 
	$N=YI_a$.
\end{definition}

\begin{definition}
	The {\bf reaction vectors} for a given reaction network $\left(\mathscr{S},\mathscr{C},\mathscr{R}\right)$ are the elements of the set $\left\{{C_j} - {C_i} \in \mathbb{R}^m|\left( {{C_i},{C_j}} \right) \in \mathscr{R}\right\}.$
\end{definition}

\begin{definition}
	The {\bf stoichiometric subspace} of a reaction network $\left(\mathscr{S},\mathscr{C},\mathscr{R}\right)$, denoted by $S$, is the linear subspace of $\mathbb{R}^m$ given by $$S = {\rm{span}}\left\{ {{C_j} - {C_i} \in \mathbb{R}^m|\left( {{C_i},{C_j}} \right) \in \mathscr{R}} \right\}.$$ The {\bf rank} of the network, denoted by $s$, is given by $s=\dim S$. The set $\left( {x + S} \right) \cap \mathbb{R}_{ \ge 0}^m$ is said to be a {\bf stoichiometric compatibility class} of $x \in \mathbb{R}_{ \ge 0}^m$.
\end{definition}

\begin{definition}
	Two vectors $x, x^{*} \in {\mathbb{R}^m}$ are {\bf stoichiometrically compatible} if $x-x^{*}$ is an element of the stoichiometric subspace $S$.
\end{definition}

%\begin{definition}
%A vector $x \in {\mathbb{R}^m}$ is {\bf stoichiometrically compatible} with the stoichiometric subspace $S$ if 
%$sign\left( {{x_s}} \right) = sign\left( {{\sigma _s}} \right)\forall s \in \mathscr{S}$
%for some $\sigma \in S$.
%\end{definition}

\begin{definition}
	Let $\mathscr{N} = \left(\mathscr{S},\mathscr{C},\mathscr{R}\right)$ be a CRN. The {\bf{map of complexes}} $Y:{\mathbb{R}^n} \to \mathbb{R}_{ \ge 0}^m$
	maps the basis vector ${\omega _y}$ to the complex $y \in \mathscr{C}$. The {\bf{incidence map}} ${I_a}:{\mathbb{R}^r} \to {\mathbb{R}^n}$
	is defined by mapping for each reaction $R_i : y \to y' \in \mathscr{R}$, the basis vector ${\omega _{R_i}}$ (or $\omega_i$) to the vector ${\omega _{y'}} - {\omega _y} \in \mathscr{C}$. The {\bf{stoichiometric map}} $N : \mathbb{R}^r \to \mathbb{R}^m$ is defined as $N=Y \circ I_a$.
\end{definition}

The reactant map $\rho$ induces a further useful mapping called {\bf{reactions map}}. The
reactions map $\rho'$ associates the coordinate of reactant complexes to the coordinates of all
reactions of which the complex is a reactant, i.e.,
$\rho' : \mathbb{R}^n \to \mathbb{R}^r$ such that $f : \mathscr{C} \to \mathbb{R}$ mapped to $f \circ \rho$.

We also associate three important linear maps to a positive element $k \in \mathbb{R}^r$. The {\bf{$k$-diagonal map}} ${\rm{diag}}(k)$ maps $\omega_i$ to $k_i\omega_i$ for $R_i \in \mathscr{R}$. The {\bf{$k$-incidence map}} $I_k$ is defined
as the composition ${\rm{diag}}(k) \circ \rho'$. Lastly, the {\bf{Laplacian map}} $A_k : \mathbb{R}^n \to \mathbb{R}^n$ is given by $A_k = I_a \circ I_k$.

CRNs can be seen as graphs. Complexes can be viewed as vertices and reactions as edges. If there is a path between two vertices $C_i$ and $C_j$, then they are said to be {\bf connected}. If there is a directed path from vertex $C_i$ to vertex $C_j$ and vice versa, then they are said to be {\bf strongly connected}. If any two vertices of a subgraph are {\bf (strongly) connected}, then the subgraph is said to be a {\bf (strongly) connected component}. The (strong) connected components are precisely the {\bf (strong) linkage classes} of a chemical reaction network. The maximal strongly connected subgraphs where there are no edges from a complex in the subgraph to a complex outside the subgraph is said to be the {\bf terminal strong linkage classes}.
We denote the number of linkage classes, the number of strong linkage classes, and the number of terminal strong linkage classes by $l, sl,$ and $t$, respectively.
A chemical reaction network is said to be {\bf weakly reversible} if $sl=l$, and it is said to be {\bf $t$-minimal} if $t=l$.

\begin{definition}
	For a chemical reaction network, the {\bf deficiency} is given by $\delta=n-l-s$ where $n$ is the number of complexes, $l$ is the number of linkage classes, and $s$ is the dimension of the stoichiometric subspace $S$.
\end{definition}

\subsection{Fundamentals of chemical kinetic systems}

\begin{definition}
	A {\bf kinetics} $K$ for a reaction network $\mathscr{N}$ is an assignment to each reaction $j: y \to y' \in \mathscr{R}$ of a rate function ${K_j}:{\Omega _K} \to {\mathbb{R}_{ \ge 0}}$ such that $\mathbb{R}_{ > 0}^m \subseteq {\Omega _K} \subseteq \mathbb{R}_{ \ge 0}^m$, $c \wedge d \in {\Omega _K}$ if $c,d \in {\Omega _K}$, and ${K_j}\left( c \right) \ge 0$ for each $c \in {\Omega _K}$.
	Furthermore, it satisfies the positivity property: supp $y$ $\subset$ supp $c$ if and only if $K_j(c)>0$.
	The system $\left(\mathscr{N},K\right)$ is called a {\bf chemical kinetic system}.
\end{definition}

\begin{definition}
	The {\bf species formation rate function} (SFRF) of a chemical kinetic system is given by $f\left( x \right) = NK(x)= \displaystyle \sum\limits_{{C_i} \to {C_j} \in \mathscr{R}} {{K_{{C_i} \to {C_j}}}\left( x \right)\left( {{C_j} - {C_i}} \right)}.$
\end{definition}
The ODE or dynamical system of a chemical kinetics system is $\dfrac{{dx}}{{dt}} = f\left( x \right)$. An {\bf equilibrium} or {\bf steady state} is a zero of $f$.

\begin{definition}
	The {\bf set of positive equilibria} of a chemical kinetic system $\left(\mathscr{N},K\right)$ is given by ${E_ + }\left(\mathscr{N},K\right)= \left\{ {x \in \mathbb{R}^m_{>0}|f\left( x \right) = 0} \right\}.$
\end{definition}

A chemical reaction network is said to admit {\bf multiple equilibria} if there exist positive rate constants such that the ODE system admits more than one stoichiometrically compatible equilibria.

Analogously, the {\bf set of complex balanced equilibria} \cite{HornJackson} is given by 
\[{Z_ + }\left(\mathscr{N},K\right) = \left\{ {x \in \mathbb{R}_{ > 0}^m|{I_a} \cdot K\left( x \right) = 0} \right\} \subseteq {E_ + }\left(\mathscr{N},K\right).\]
A positive vector $c \in \mathbb{R}^m$ is complex balanced if $K\left( c \right)$ is contained in $\ker{I_a}$, and a chemical kinetic system is complex balanced if it has a complex balanced equilibrium.

\begin{definition}
	A kinetics $K$ is a {\bf power-law kinetics} (PLK) if 
	${K_i}\left( x \right) = {k_i}{{x^{{F_{i}}}}} $  $\forall i =1,...,r$ where ${k_i} \in {\mathbb{R}_{ > 0}}$ and ${F_{ij}} \in {\mathbb{R}}$. The power-law kinetics is defined by an $r \times m$ matrix $F$, called the {\bf kinetic order matrix} and a vector $k \in \mathbb{R}^r$, called the {\bf rate vector}.
\end{definition}
If the kinetic order matrix is the transpose of the molecularity matrix, then the system becomes the well-known {\bf mass action kinetics (MAK)}.

\begin{definition}
	A PLK system has {\bf reactant-determined kinetics} (of type PL-RDK) if for any two reactions $i, j$ with identical reactant complexes, the corresponding rows of kinetic orders in $F$ are identical. That is, ${f_{ik}} = {f_{jk}}$ for $k = 1,2,...,m$. A PLK system has {\bf non-reactant-determined kinetics} (of type PL-NDK) if there exist two reactions with the same reactant complexes whose corresponding rows in $F$ are not identical.
\end{definition}

In \cite{arceo2015}, a discussion on several sets of kinetics of a network is given. They identified two main sets: the {\bf{complex factorizable kinetics (CFK)}} and its complement, the {\bf{non-complex factorizable kinetics (NFK)}}.

\begin{definition}
	A chemical kinetics $K$ is {\bf{complex factorizable}} (CF) if there is a $k \in \mathbb{R}_{ > 0}^r$ and a mapping ${\Psi _K}:{\Omega _K} \to \mathbb{R}_{ \ge 0}^n$ such that $\mathbb{R}_{ > 0}^m \subseteq {\Omega _K} \subseteq \mathbb{R}_{ \ge 0}^m$ and $K = {I_K} \circ {\Psi _K}$. We call ${\Psi _K}$ a {\bf{factor map}} of $K$.
\end{definition}

We obtain $K = {\rm{diag}}\left( k \right) \circ \rho ' \circ {\Psi _K}$ by the use of the definition of the $k$-incidence map $I_k$. This
implies that a complex factorizable kinetics $K$ has a decomposition into diagonal matrix of
rate constants and an interaction map $\rho ' \circ {\Psi _K}:R_{ \ge 0}^m \to R_{ \ge 0}^r$. The values of the interaction
map are ``reaction-determined'' in the sense that they are determined by the values on the reactant complexes.

Complex factorizable kinetics generalizes the key structural property of MAK. The complex formation rate function decomposes as $g = {A_k} \circ {\Psi _K}$ while the species formation rate function factorizes as $f = Y \circ {A_k} \circ {\Psi _K}$. The complex factorizable kinetic systems in the set of power-law kinetic systems are precisely the PL-RDK systems \cite{arceo2015}.

\subsection{Review of decomposition theory}
\label{sect:decomposition}

This subsection introduces important definitions and earlier results from the decomposition theory of chemical reaction networks. A more detailed discussion can be found in \cite{FAML2020}.

\begin{definition}
	A {\bf decomposition} of $\mathscr{N}$ is a set of subnetworks $\{\mathscr{N}_1, \mathscr{N}_2,...,\mathscr{N}_k\}$ of $\mathscr{N}$ induced by a partition $\{\mathscr{R}_1, \mathscr{R}_2,...,\mathscr{R}_k\}$ of its reaction set $\mathscr{R}$. 
\end{definition}

We denote a decomposition by 
$\mathscr{N} = \mathscr{N}_1 \cup \mathscr{N}_2 \cup ... \cup \mathscr{N}_k$
since $\mathscr{N}$ is a union of the subnetworks in the sense of \cite{ghms2019}. It also follows immediately that, for the corresponding stoichiometric subspaces, 
${S} = {S}_1 + {S}_2 + ... + {S}_k$.

The following important concept of independent decomposition was introduced by Feinberg in \cite{feinberg12}.

\begin{definition}
	A network decomposition $\mathscr{N} = \mathscr{N}_1 \cup \mathscr{N}_2 \cup ... \cup \mathscr{N}_k$  is {\bf independent} if its stoichiometric subspace is a direct sum of the subnetwork stoichiometric subspaces.
\end{definition}

In \cite{fortun2}, it was shown that for any independent decomposition, the following holds $$\delta \le \delta_1 +\delta_2 ... +\delta_k.$$

\begin{definition} \cite{FAML2020}
	A decomposition $\mathscr{N} = \mathscr{N}_1 \cup \mathscr{N}_2 \cup ... \cup \mathscr{N}_k$ with $\mathscr{N}_i = (\mathscr{S}_i,\mathscr{C}_i,\mathscr{R}_i)$ is a {\bf $\mathscr{C}$-decomposition} if for each pair of distinct $i$ and $j$, $\mathscr{C}_i$ and $\mathscr{C}_j$ are disjoint.
\end{definition}

\begin{definition}
	A decomposition of CRN $\mathscr{N}$ is {\bf incidence-independent} if the incidence map $I_a$ of $\mathscr{N}$ is the direct sum of the incidence maps of the subnetworks. It is {\bf bi-independent} if it is both independent and incidence-independent.
\end{definition}

We can also show incidence-independence by satisfying the equation $n - l = \sum {\left( {{n_i} - {l_i}} \right)}$, where $n_i$ is the number of complexes and $l_i$ is the number of linkage classes, in each subnetwork $i$.

In \cite{FAML2020}, it was shown that for any incidence-independent decomposition, $$\delta \ge \delta_1 +\delta_2 ... +\delta_k.$$ In addition, $\mathscr{C}$-decompositions form a subset of incidence-independent decompositions.

Feinberg established the following relation between an independent decomposition and the set of positive equilibria of a kinetics on the network.

\begin{theorem} (Feinberg Decomposition Theorem \cite{feinberg12})
	\label{feinberg:decom:thm}
	Let $\{\mathscr{R}_1, \mathscr{R}_2,...,\mathscr{R}_k\}$ be a partition of a CRN $\mathscr{N}$ and let $K$
	be a kinetics on $\mathscr{N}$. If $\mathscr{N} = \mathscr{N}_1 \cup \mathscr{N}_2 \cup ... \cup \mathscr{N}_k$ is the network decomposition of the partition and ${E_ + }\left(\mathscr{N}_i,{K}_i\right)= \left\{ {x \in \mathbb{R}^m_{>0}|N_iK_i(x) = 0} \right\}$, then
	\[{E_ + }\left(\mathscr{N}_1,K_1\right) \cap {E_ + }\left(\mathscr{N}_2,K_2\right) \cap ... \cap {E_ + }\left(\mathscr{N}_k,K_k\right) \subseteq  {E_ + }\left(\mathscr{N},K\right).\]
	If the network decomposition is independent, then equality holds.
\end{theorem}

The analogue of Feinberg's 1987 result for incidence-independent decompositions and complex balanced equilibria is shown in \cite{FAML2020}:

\begin{theorem} (Theorem 4 \cite{FAML2020})
	\label{decomposition:thm:2}
	Let $\mathscr{N}=(\mathscr{S},\mathscr{C},\mathscr{R})$ be a a CRN and $\mathscr{N}_i=(\mathscr{S}_i,\mathscr{C}_i,\mathscr{R}_i)$ for $i = 1,2,...,k$ be the subnetworks of a decomposition.
	Let $K$ be any kinetics, and $Z_+(\mathscr{N},K)$ and $Z_+(\mathscr{N}_i,K_i)$ be the set of complex balanced equilibria of $\mathscr{N}$ and $\mathscr{N}_i$, respectively. Then
	\begin{itemize}
		\item[i.] ${Z_ + }\left(\mathscr{N}_1,K_1\right) \cap {Z_ + }\left(\mathscr{N}_2,K_2\right) \cap ... \cap {Z_ + }\left(\mathscr{N}_k,K_k\right) \subseteq  {Z_ + }\left(\mathscr{N},K\right)$.\\
		If the decomposition is incidence independent, then
		\item[ii.] ${Z_ + }\left( {\mathscr{N},K} \right) = {Z_ + }\left(\mathscr{N}_1,K_1\right) \cap {Z_ + }\left(\mathscr{N}_2,K_2\right) \cap ... \cap {Z_ + }\left(\mathscr{N}_k,K_k\right)$, and
		\item[iii.] ${Z_ + }\left( {\mathscr{N},K} \right) \ne \varnothing$ implies ${Z_ + }\left( {\mathscr{N}_i,K_i} \right) \ne \varnothing$ for each $i=1,...,k$.
	\end{itemize}
\end{theorem}

The following converse of Theorem \ref{decomposition:thm:2} iii holds for a subset of incidence-independent decompositions with any kinetics.

\begin{theorem} (Theorem 5 \cite{FAML2020})
	Let $\mathscr{N} = \mathscr{N}_1 \cup \mathscr{N}_2 \cup ... \cup \mathscr{N}_k$ be a weakly reversible $\mathscr{C}$-decomposition of a chemical kinetic system $(\mathscr{N},K)$. If ${Z_ + }\left( {\mathscr{N}_i,K_i} \right) \ne \varnothing$ for each $i=1,...,k$, then ${Z_ + }\left( {\mathscr{N},K} \right) \ne \varnothing$.
\end{theorem}

\section{A poly-PL approach to Hill-type kinetics}
\label{positive:equilibria:hill}

\subsection{Some basic properties of poly-PL systems}
In this section, we associate a unique poly-PL system to an HTK system. Its key property, that its sets of positive equilibria and of complex balanced equilibria coincide with those of the Hill-type system, is the basis of our analysis. We compare the approach with a simpler procedure for determining the equilibria sets in order to clarify its significance for deriving kinetic system properties of the original system. We begin with a brief review of concepts and results about poly-PL kinetic (denoted by PYK) systems. For further details on these systems, we refer to \cite{FTJM2020,TMMN2020}.

We recall the definition of a poly-PL kinetics from Talabis et al. \cite{TMMN2020} and introduce some additional notation:

\begin{definition}
	A kinetics $K: \mathbb{R}_{ > 0}^m \to \mathbb{R}^r$ is a {\bf poly-PL kinetics} if
	$${K_i}\left( x \right) = {k_i}\left( {{a_{i1}}{x^{{F_{i1}}}} + ... + {a_{ij}}{x^{{F_{ij}}}}} \right)\ \ \ \forall i \in \left\{ {1,...,r} \right\}$$
	written in lexicographic order with $k_i>0, a_{ij} \ge 0$, $F_i \in \mathbb{R}^m$ and $j \in \{1,...,h_i\}$ where $h_i$ is the number of terms in reaction $i$.
	If $h = \max {h_i}$, we normalize the length of each kinetics to $h$ by replacing the last term with $(h - h_i+1)$ terms with $\dfrac{1}{{\left( {h - {h_i}+1} \right)}}{x^{{F_{i{h_i}}}}}$.
	We call this the {\bf canonical PL-representation} of a poly-PL kinetics.
\end{definition}
We set ${K_j}\left( x \right): = {k_{ij}}{x^{{F_{ij}}}}$ with ${k_{ij}} = {k_i}{a_{ij}}$ where $i=1,...,r$ for each $j = 1,...,h$. $(\mathscr{N}, K_j)$ is a PLK system with kinetic order matrix $F_j$. Also, we let the matrix $A:=A_{ij} = \left( a_{ij} \right)$. We denote with PY-NIK the set of poly-PL kinetics with only nonnegative kinetic orders.

It follows from the canonical PL-representation of a poly-PL kinetics that the ${\mathop{\rm span}\nolimits} \left\langle {{\mathop{\rm im}\nolimits} K} \right\rangle $
is contained in the sum
$\left\langle {{\mathop{\rm im}\nolimits} {K_1}} \right\rangle  + ... + \left\langle {{\mathop{\rm im}\nolimits} {K_h}} \right\rangle $.

\begin{definition}
	A poly-PL kinetic system $(\mathscr{N}, K)$ is called {\bf PL-decomposable} if the sum $\left\langle {{\mathop{\rm im}\nolimits} {K_1}} \right\rangle  + ... + \left\langle {{\mathop{\rm im}\nolimits} {K_h}} \right\rangle $  is direct.
\end{definition}

It is shown in \cite{FTJM2020} that many properties of PLK systems can be extended to the set of PL-decomposable poly-PL systems.

The {\bf{$S$-invariant termwise addition of reactions via maximal stoichiometric
		coefficients (STAR-MSC)}} method is based on the idea of using the maximal stoichiometric
coefficient (MSC) among the complexes in the CRN to construct reactions whose reactant
complexes and product complexes are different from existing ones. This is done by uniform
translation of the reactants and products to create a ``replica'' of the CRN. The method
creates $h - 1$ replicas of the original network. Hence, its transform $\mathscr{N}^*$ becomes the
union (in the sense of \cite{ghms2019}) of the replicas and the original CRN.

We now describe the STAR-MSC method. All $x=(x_1,\ldots,x_m)$ are positive vectors
since the domain in the definition of a poly-PL kinetics is $\mathbb{R}^{m}_{>0}$. 
Let $M = 1 + \max \{y_i | y \in \mathscr{C}\}$, where the second summand is the maximal stoichiometric coefficient, which is a positive integer.

For every positive integer $z$, let $z$ be identified with the vector $(z,z,\ldots,z)$ in $\mathbb{R}^{m}$.
For each complex $y \in \mathscr{C}$, form the $(h-1)$ complexes 
\[y + M, y + 2M, \ldots, y + (h -1)M.\]
Each of these complexes are different from all existing complexes and each other as shown
in the following Proposition in \cite{FTJM2020}:

\begin{proposition}
	Let $\mathscr{N}^*=(\mathscr{S}, \mathscr{C}^*, \mathscr{R}^*)$ be a STAR-MSC transform of $\mathscr{N}=(\mathscr{S}, \mathscr{C}, \mathscr{R})$,
	$\mathscr{N}^*_1:=\mathscr{N}$ and $\mathscr{N}^*_j$ is the subnetwork defined by $\mathscr{R}^*_{j-1}$ for $j \in \{2,\ldots,h\}$. Then  $|\mathscr{R}^*|=hr$ and $|\mathscr{C}^*| = hn$.
\end{proposition}
The transformation is used to extend the multistationarity algorithm in \cite{hmr2019,hmr22019} from power-law to poly-PL systems.
Further details of the STAR-MSC transformation are provided in \cite{MAGP2019,MHRM2020}.

The {\bf{rate constant-interaction map decomposable kinetics (RIDK)}} was introduced in \cite{NEML2019}.
These are chemical kinetics that have constant rates. Formally, it is a kinetics such that for each reaction $R_q$, the coordinate function $K_q : \Omega \to \mathbb{R}$ can be written in the form $K_q(x) = k_qI_{K,q}(x)$,
with $k_q \in \mathbb{R}_{>0}$ (called rate constant) and $\Omega \in \mathbb{R}^m$. We call the map $I_K : \Omega \to \mathbb{R}^r$ defined
by $I_{K,q}$ as the {\bf{interaction map}}.

A poly-PL kinetics $K$ is a rate constant-interaction decomposable (RID) kinetics, the interaction $I_K$ being the poly-PL function. By definition, an RID kinetics is called {\bf complex factorizable} \cite{NEML2019} if at reach branching point of the network, any branching reaction has the same interaction. The subset of such kinetics is denoted by PY-RDK. We have the following Proposition from \cite{FTJM2020}:

\begin{proposition}
	\label{prop:complex:factorizable}
	Let $\left\{ {{K_j}} \right\}$ be the canonical PL-representation of the poly-PL kinetic system $(\mathscr{N}, K)$.  $(\mathscr{N}, K)$ is complex factorizable if and only if for each $j$, $(\mathscr{N}, K_j)$ is a PL-RDK system and for any two branching reactions $R$ and $R'$ of a node, ${A_{Rj}} = {A_{R'j}}$.
\end{proposition}

\subsection{The associated poly-PL system of a Hill-type system}
For a Hill type kinetics $K$ and a reaction $r_k$, we first write $M_k$ and $T_k$ as products of two factors each, one of which contains the expressions for the positive kinetic orders and the other for the negative ones, i.e., ${M_k} = {M_k}^ + {M_k}^ - $ and ${T_k} = {T_k}^ + {T_k}^ - $, respectively.
We then write $\dfrac{{{M_k}^{-} }}{{{T_k}^{-} }} = \dfrac{1}{{{T_k}^{'}}}$, 
where ${T_k}^{'} = \displaystyle\prod {\left( {{d_{ki}}{x_i}^{\left| {{F_{ki}}} \right|} + 1} \right)}$.
Hence, $\dfrac{{{M_k}}}{{{T_k}}} = \dfrac{{{M_k}^{+} }}{{{T_k}^{+} {T_k}^{'}}}$. Let
$T\left( x \right): = \displaystyle \prod\limits_k {{T_k}^{+} {T_k}^{'}} $.
$T(x)$ can be rewritten as a product of distinct bi-PL factors
$${\left| {{T_l}} \right|} \in {\mathscr{F}_K}: = \left\{ {{d_{ki}} + {x_i}^{{F_{ki}}},{d_{ki}}{x_i}^{\left| {{F_{ki}}} \right|} + 1} \right\},$$
i.e., $T\left( x \right) = \displaystyle \prod\limits_l {{{\left| {{T_l}} \right|}^{\mu \left( l \right)}}} $
where $\mu(l) =$ the number of times the bi-PL occurs in $T(x)$. 

For a bi-PL ${\left| {{T_l}} \right|}$, we define $\omega (l) :=$ maximal number of occurrence in a ${T_k}^{+} {T_k}^{'}$, $k = 1,\ldots, r$.
Clearly, for each bi-PL, $\omega (l) \le \mu(l)$. We can now define the central concept for the construction:

\begin{definition}
	${\text{LCD}}\left( K \right) = \displaystyle \prod {{{\left| {{T_l}} \right|}^{\omega \left( l \right)}}}$ is the least common denominator of the Hill-type kinetics $K$.
\end{definition}

Clearly, LCD$(K)$ is a factor of $T(x)$. For each reaction $q$, ${{T_q}^{+} {T_q}^{'}}$is a factor of LCD$(K)$, i.e.,
${T_q}^{+} {T_q}^{'}{L_q} = {\text{LCD}}\left( K \right)$ or $\dfrac{1}{{{T_q}^{+} {T_q}^{'}}} = \dfrac{{{L_q}}}{{{\text{LCD}}\left( K \right)}}$.
Hence, ${K_q}\left( x \right) = {k_q}\dfrac{{{M_q}^{+} \left( x \right){L_q}\left( x \right)}}{{{\text{LCD}}\left( K \right)}}$.
We now formulate the new definition of $K_{\text{PY}}$:

\begin{definition}
	The kinetics given by ${K_{{\text{PY}},q}}\left( x \right): = {k_q}{M_q}^ + \left( x \right){L_q}\left( x \right)$ is called the associated poly-PL kinetics of an HTK given by
	${K_q}\left( x \right) = {k_q}\dfrac{{{M_q}\left( x \right)}}{{{T_q}\left( x \right)}}$.
\end{definition}

The key properties of $K_\text{PY}$ are derived in the following theorem:
\begin{theorem}
	For any HTK system $(\mathscr{N}, K)$ we have:
	\begin{itemize}
		\item[i.] $Z_+(\mathscr{N}, K) = Z_+(\mathscr{N}, K_\text{PY})$
		\item[ii.] $E_+(\mathscr{N}, K) = E_+(\mathscr{N}, K_\text{PY})$
		\item[iii.] Like $K$, $K_{\text{PY}}$ is defined on the whole nonnegative orthant, i.e., $K_{\text{PY}} \in$ PY-NIK.
		%\item[iii.] the number of terms in $K_{\text{PY},q}(x)$ is $h_q = 2^{(r -1)|{\rm supp} \ d_q|}$ , therefore, $h = 2^{(r -1)d_{\rm max}}$, where $d_{\rm max} = {\rm max}_q |{\rm supp} \ d_q|$
	\end{itemize}
	\label{theorem:equilibriasets}
\end{theorem}

\begin{proof}
	For a reaction $q:{y_q} \to y{'_q}$, we denote the characteristic functions of the reactant and product complexes by ${\omega _q}$ and $\omega {'_q}$, respectively.
	To show (i.), consider $K(x)$'s CFRF (complex formation rate function) and we have
	\[g\left( x \right) = \sum\limits_{q \in \mathscr{R}} {{k_q}\frac{{{M_q}\left( x \right)}}{{{T_q}\left( x \right)}}\left( {\omega {'_q} - {\omega _q}} \right)}  = \sum\limits_{q \in \mathscr{R}} {{k_q}\frac{{M_q}^{+} \left( x \right){L_q}\left( x \right)}{{\text{LCD}}\left( K \right)}\left( {\omega {'_q} - {\omega _q}} \right)} \]
	\[ = \frac{1}{{\text{LCD}}\left( K \right)}\sum\limits_{q \in \mathscr{R}} {{k_q}{{M_q}^{+} \left( x \right){L_q}\left( x \right)}\left( {\omega {'_q} - {\omega _q}} \right)}  = \frac{1}{{\text{LCD}}\left( K \right)}{g_{{\text{PY}}}}\left( x \right)\]
	where ${g_{{\text{PY}}}}\left( x \right)$ is $K_\text{PY}$'s CFRF.
	Since for $x > 0$, $1/ [{\text{LCD}}\left( K \right)] > 0$, we obtain $g(x) = 0$ if and only if  ${g_{{\text{PY}}}}\left( x \right) = 0$, which is the claim.
	Analogously, for the species rate formation functions, we obtain for $x > 0$, $f(x) = 0$ if and only if  ${f_{{\rm{PY}}}}\left( x \right) = 0$ and hence (ii). Claim (iii) follows directly from the construction.
\end{proof}

We denote the canonical PL-representation of $K_\text{PY} = K_{\text{PY},1} + ... + K_{\text{PY},h}$, where $h \le 2^{(r -1)d_{\rm max}}$ . The kinetic order matrix of $K_{\text{PY},j}$ with $F_j$. As usual, the kinetic order matrix of $M := (M_q)$ is $F$.

\begin{remark}
	If $T(x)$ is constant, then $(\mathscr{N}, K)$ is linearly conjugate to $(\mathscr{N}, K_\text{PY})$. A trivial condition for this is when the kinetic order matrix $F$ is the zero matrix, in which case both kinetics are also constant functions.
\end{remark}

Using $K_{\text{PY}}$, we obtain an analogue of a result in Section 15 of \cite{WIFE2013} concerning the relationship of multistationarity in a Hill-type system and in an associated power-law system:

\begin{proposition}
	If the Hill-type system $(\mathscr{N},K)$ has multiple positive steady states in a stoichiometric class $x + S$, then the STAR-MSC transform $(\mathscr{N}^*,K_{\text{PY}}^*)$ of its associated poly-PL system $(\mathscr{N},K_{\text{PY}})$ also has multiple positive steady states in the same stoichiometric class.
\end{proposition}

\begin{proof}
	$(\mathscr{N},K)$ and $(\mathscr{N},K_{\text{PY}})$ have identical positive equilibria sets as shown earlier. Since STAR-MSC is a dynamic equivalence, the positive equilibria sets of $(\mathscr{N},K_{\text{PY}})$ and $(\mathscr{N}^*,K_{\text{PY}}^*)$ are also identical. Since STAR-MSC also leaves the stoichiometric subspace invariant, we obtain the claim. 
\end{proof}

We say that the result is only analogous since $\mathscr{N}^*$ is formally not identical to $\mathscr{N}$, though ``essentially'' it is $\mathscr{N}$, being the union of $\mathscr{N}$ and $(h-1)$ replicas.

We now introduce the following definition:

\begin{definition}
	HT-RDK := HTK $\cap$ CFK
\end{definition}

A further important property of the associated poly-PL kinetics is the following:

\begin{proposition}
	$K \in {\rm{HT}{\text{-}\rm{RDK}}} \Leftrightarrow {K_{\rm{PY}}} \in {\rm{PY}{\text{-}\rm{RDK}}}$
	\label{prop:KPY}
\end{proposition}

\begin{proof}
	For reactions $q$ and $q'$ with the same reactant, we have $${I_{K,q}} = \dfrac{{{M_q}}}{{{T_q}}} = \dfrac{{{M_{q'}}}}{{{T_{q'}}}} = {I_{K,q'}}.$$ Multiplying the equations by ${\text{LCD}}\left( K \right)$, we obtain ${M_q}^{+} {L_q} = {M_{q'}}^{+} {L_{q'}}$, which is equivalent to the equality of the interaction functions of $K_{{\rm{PY}},q}$ and $K_{{\rm{PY}},{q'}}$. 
\end{proof}

In \cite{AJLM2017}, Arceo et al. defined HT-RDKD as the set of HT kinetics such that for any two reactions $q$ and $q'$ with the same reactant, $F_q = F_{q'}$ and $D_q = D_{q'}$, where $F$ and $D$ are the kinetic order matrix (of the numerator) and the dissociation matrix respectively. In the same paper, it was shown that HT-RDKD is contained in HTK $\cap$ CFK. Without the assumption of ${\rm supp} (D_q) = {\rm supp} (F_q)$, we would have obtained the following as an example of HTK $\cap$ CFK $\backslash$ HT-RDKD.

\noindent Consider the CRN $\begin{array}{*{20}{c}}
	{{R_1}:{X_1} + {X_2} \to {X_1}}\\
	{{R_2}:{X_1} + {X_2} \to {X_2}}
\end{array}$ with kinetics defined by
$F = \left[ {\begin{array}{*{20}{c}}
		1&0\\
		1&1
\end{array}} \right]$ and
$D = \left[ {\begin{array}{*{20}{c}}
		1&0\\
		1&0
\end{array}} \right]$.
Then
\[K\left( x \right) = \left[ {\begin{array}{*{20}{c}}
		{\dfrac{{{M_1}}}{{{T_1}}}}\\
		{\dfrac{{{M_2}}}{{{T_2}}}}
\end{array}} \right] = \left[ {\begin{array}{*{20}{c}}
		{\dfrac{{{x_1}}}{{1 + {x_1}}}}\\
		{\dfrac{{{x_1}{x_2}}}{{\left( {1 + {x_1}} \right){x_2}}}}
\end{array}} \right] = \left[ {\begin{array}{*{20}{c}}
		{\dfrac{{{x_1}}}{{1 + {x_1}}}}\\
		{\dfrac{{{x_1}}}{{1 + {x_1}}}}
\end{array}} \right]
\]
with rate constants equal to 1.
Note that $\dfrac{{{M_1}}}{{{T_1}}} = \dfrac{{{M_2}}}{{{T_2}}}$, but ${M_1}$ and ${M_2}$ are not necessarily equal.

\begin{proposition}
	HT-RDKD = HT-RDK assuming ${\rm supp} (D_q) = {\rm supp} (F_q)$.
\end{proposition}

\begin{proof}
	We already mentioned that HT-RDKD $\subseteq$ HT-RDK. We are left with the other containment. Note that for a reaction $q$, \[{K_q}\left( x \right) = k_q \frac{{{M_q}\left( x \right)}}{{{T_q}\left( x \right)}} = {k_q}\frac{{\prod\limits_{i = 1}^m {{x_i}^{{F_{qi}}}} }}{{\prod\limits_{i = 1}^m {\left( {{d_{qi}} + {x_i}^{{F_{qi}}}} \right)} }} = {k_q}\prod\limits_{i = 1}^m {\frac{{{x_i}^{{F_{qi}}}}}{{{d_{qi}} + {x_i}^{{F_{qi}}}}}} .\]
	Consider arbitrary $x_i$ for a reaction $q$.\\
	{\bf{Case 1:}} Suppose an entry $d_{qi}=0$. Since ${\rm supp} (D_q) = {\rm supp} (F_q)$, then $f_{qi}=0$. Thus, the corresponding factor is \[\frac{{{x_i}^{{F_{qi}}}}}{{{d_{qi}} + {x_i}^{{F_{qi}}}}} = \frac{{{x_i}^0}}{{0 + {x_i}^0}} = \frac{1}{1}= 1,\]
	which has no effect on both $\dfrac{{{M_q}\left( x \right)}}{{{T_q}\left( x \right)}}$ and $M_q(x)$.\\
	{\bf{Case 2:}} Suppose an entry $d_{qi}\ne 0$.
	%Again ${\rm supp} (D_q) = {\rm supp} (F_q)$, which implies $f_{ij} \ne 0$.
	It follows that the corresponding factor
	$\dfrac{{{x_i}^{{F_{qi}}}}}{{{d_{qi}} + {x_i}^{{F_{qi}}}}}$
	stays as it is (i.e., no necessary cancellation).\\
	%(i.e., no cancellation between the numerator and the denominator)
	Therefore, for two reactions $q$ and $q'$,
	\[\frac{{{M_q}\left( x \right)}}{{{T_q}\left( x \right)}} = \frac{{{M_{q'}}\left( x \right)}}{{{T_{q'}}\left( x \right)}} \Rightarrow {\text{ both }} {M_q}\left( x \right) = {M_{q'}}\left( x \right) {\text{ and }} {D_q} = {D_{q'}}.\] 
	
\end{proof}

\subsection{Comparison with a case-specific procedure for determining positive equilibria}
It is instructive to compare the association of the poly-PL kinetic system $(\mathscr{N}, K_\text{PY})$ to an HTK system $(\mathscr{N}, K)$ for any set of rate constants with a procedure for determining the equilibria set of $(\mathscr{N}, K)$ for a given set of rate constants. This procedure also uses factoring out common denominators, but on a species-dependent basis.

For each species $X_i$, for $i = 1,...,m$, we define $${\mathscr{R}_i} := \left\{ {y \to y':{X_i} \in {\mathop{\text{supp } }\nolimits}  y \cup {\mathop{\text{supp } }\nolimits}  y'} \right\},$$ i.e., the set of reactions where $X_i$ occurs in the reactant or product complex. Furthermore, define
${T^{\left( i \right)}}\left( x \right) = \prod\limits_{k \in {\mathscr{R}_i}} {{T_k}\left( x \right)} $ and for a reaction $q$, ${\mathscr{R}_{i,q}} = {\mathscr{R}_i}\backslash \left\{ q \right\}$, ${T^{\left( i \right)}}_{\left( q \right)}\left( x \right) = \prod\limits_{k \in {\mathscr{R}_{i,q}}} {{T_k}\left( x \right)} $, ${P^{\left( i \right)}}_q\left( x \right) = {M_q}\left( x \right){T^{\left( i \right)}}_{\left( q \right)}\left( x \right)$. The $i$-th coordinate of the SFRF (= RHS of $i$-th ODE) is given by
\begin{equation*}
	%equation with label equation* without label
	\begin{split}
		{f_i}\left( x \right) &= \sum\limits_{q \in {\mathscr{R}_i}} {\frac{{{k_q}{M_q}\left( x \right)}}{{{T_q}\left( x \right)}}\left( {y{'_q} - {y_q}} \right)}  = \sum\limits_{q \in {\mathscr{R}_i}} {\frac{{{k_q}{M_q}\left( x \right){T^{\left( i \right)}}_{\left( q \right)}\left( x \right)}}{{{T^{\left( i \right)}}\left( x \right)}}\left( {y{'_q} - {y_q}} \right)} \\
		&=\frac{1}{{{T^{\left( i \right)}}\left( x \right)}}\sum\limits_{q \in {\mathscr{R}_i}} {{k_q}{P^{\left( i \right)}}_q\left( x \right)\left( {y{'_q} - {y_q}} \right)}. 
	\end{split}
\end{equation*}
Writing ${\widetilde{f}_i}\left( x \right)$ for the last sum, since $\dfrac{1}{{{T^{\left( i \right)}}\left( x \right)}} > 0$ for $x > 0$, we have
\[{{f}_i}\left( x \right) =0 \Leftrightarrow {\widetilde{f}_i}\left( x \right)=0.\]

Since only the reactions in $\mathscr{R}_i$ occur in them, the ${P^{\left( i \right)}}_q\left( x \right)$ expressions are generally much shorter than those in $K_\text{PY}(x)$, so they are much more practical for computing the equilibria sets. However, since a reaction will in general occur in more than one coordinate function $\left( { \Leftrightarrow {\mathscr{R}_i} \cap {\mathscr{R}_{i'}} \ne \varnothing {\text{ for }}i \ne i'} \right)$, it will be assigned different
${{P^{\left( i \right)}}_q\left( x \right)}$ expressions. This means that the ODE system 
$\dfrac{{d{x_i}}}{{dt}} =$
${\widetilde{f}_i}\left( x \right)$ is not generated by a chemical kinetic system since a kinetics assigns only one rate function to a reaction. Hence, this set of ODEs cannot be used to derive {\underline{kinetic system}} properties for the HTK system $(\mathscr{N}, K)$. 

\subsection{Examples}
The following are examples of Hill-type kinetic systems and their associated poly-PL kinetic (PYK) systems. Note that the equilibria sets of these systems coincide.
\begin{example}
	Let $(\mathscr{N}, K)$ with reactions ${R_1}:{X_1} \to {X_2}$ and ${R_2}:{X_2} \to {X_1}$, and $F = \left[ {\begin{array}{*{20}{c}}
			1&0\\
			0&1
	\end{array}} \right]$ and $D = \left[ {\begin{array}{*{20}{c}}
			1&0\\
			0&1
	\end{array}} \right]$. Hence, ${I_a} = \left[ {\begin{array}{*{20}{c}}
			{ - 1}&1\\
			1&{ - 1}
	\end{array}} \right]$ and $K\left( x \right) = \left[ {\begin{array}{*{20}{c}}
			{{k_1}\dfrac{{{x_1}}}{{1 + {x_1}}}}\\
			{{k_2}\dfrac{{{x_2}}}{{1 + {x_2}}}}
	\end{array}} \right]$. 
	It follows that ${K_{\text {PY}}}\left( x \right) = \left[ {\begin{array}{*{20}{c}}
			{{k_1}{x_1}\left( {1 + {x_2}} \right)}\\
			{{k_2}{x_2}\left( {1 + {x_1}} \right)}
	\end{array}} \right]$.
	We now solve the equations
	$${I_a}K\left( x \right) = \left[ {\begin{array}{*{20}{c}}
			{ - 1}&1\\
			1&{ - 1}
	\end{array}} \right]\left[ {\begin{array}{*{20}{c}}
			{{k_1}\dfrac{{{x_1}}}{{1 + {x_1}}}}\\
			{{k_2}\dfrac{{{x_2}}}{{1 + {x_2}}}}
	\end{array}} \right] = \left[ {\begin{array}{*{20}{c}}
			0\\
			0
	\end{array}} \right] {\text {\ and}}$$
	$${I_a}{K_{\text {PY}}}\left( x \right) = \left[ {\begin{array}{*{20}{c}}
			{ - 1}&1\\
			1&{ - 1}
	\end{array}} \right]\left[ {\begin{array}{*{20}{c}}
			{{k_1}{x_1}\left( {1 + {x_2}} \right)}\\
			{{k_2}{x_2}\left( {1 + {x_1}} \right)}
	\end{array}} \right] = \left[ {\begin{array}{*{20}{c}}
			0\\
			0
	\end{array}} \right].$$
	We can easily show that
	\begin{align*}
		{Z_ + }\left( {\mathscr{N},K} \right) &= {Z_ + }\left( {\mathscr{N},{K_\text{PY}}} \right)\\
		&= \left\{ {\left( {t,\frac{{{k_1}t}}{{{k_2}\left( {1 + t} \right) - {k_1}t}}} \right):{k_1},{k_2},t > 0,{k_1}t \ne {k_2}\left( {1 + t} \right)} \right\}.
	\end{align*}
	Also, $NK(x) = {I_a}K\left( x \right)$ and $N{K_\text{PY}}\left( x \right) = {I_a}{K_\text{PY}}\left( x \right)$ so $${E_ + }\left( {\mathscr{N},K} \right) = {E_ + }\left( {\mathscr{N},{K_\text{PY}}} \right).$$
	Note that ${\widetilde{f}_i}\left( x \right)=f_{\text {PY}}(x)$ for $i=1,2$ since ${\mathscr{R}_1}={\mathscr{R}_2}={\mathscr{R}}$.
\end{example}

\begin{example}
	\label{example:running1}
	Let $(\mathscr{N}, K)$ with the following assumptions:
	\[\begin{array}{*{20}{c}}
		{{R_1}:{X_1} \to {X_1} + {X_2}}\\
		{{R_2}:{X_1} + {X_2} \to {X_3}}\\
		{{R_3}:{X_3} \to {X_1}{\rm{ \ \ \ \ \ \ \ }}}
	\end{array}{\rm{\ \ \ \ \ \ }}F = \left[ {\begin{array}{*{20}{c}}
			1&0&0\\
			1&0&0\\
			0&0&1
	\end{array}} \right]{\rm{ \ \ \ \ }}D = \left[ {\begin{array}{*{20}{c}}
			1&0&0\\
			1&0&0\\
			0&0&1
	\end{array}} \right].\]
	We have
	$Y = \left[ {\begin{array}{*{20}{c}}
			1&1&0\\
			0&1&0\\
			0&0&1
	\end{array}} \right]$,
	${I_a} = \left[ {\begin{array}{*{20}{c}}
			{ - 1}&0&1\\
			1&{ - 1}&0\\
			0&1&{ - 1}
	\end{array}} \right]$
	and consider
	$K\left( x \right) = \left[ {\begin{array}{*{20}{c}}
			{\dfrac{{{x_1}}}{{1 + {x_1}}}}\\
			{\dfrac{{{x_1}}}{{ {1 + {x_1}}}}}\\
			{\dfrac{{{x_3}}}{{1 + {x_3}}}}
	\end{array}} \right].$
	Then,
	\[{K_{{\rm{PY}}}}\left( x \right) = 
	\left[ {\begin{array}{*{20}{c}}
			{{x_1}\left( {1 + {x_3}} \right)}\\
			{ {{x_1}} \left( {1 + {x_3}} \right)}\\
			{{x_3}\left( {1 + {x_1}} \right)}
	\end{array}} \right]=
	%\]
	%\[
	\left[ {\begin{array}{*{20}{c}}
			{{x_1} + {x_1}{x_3} }\\
			{{x_1}  + {x_1}{x_3} }\\
			{{x_3} + x_1 x_3}
	\end{array}} \right].\]
	We can verify that 
	\[{Z_ + }\left( {\mathscr{N},K} \right) = {Z_ + }\left( {\mathscr{N},{K_\text{PY}}} \right) =\left\{ {\left( {s,t,s} \right):s,t > 0} \right\} {\text{and}}\]
	\[{E_ + }\left( {\mathscr{N},K} \right) = {E_ + }\left( {\mathscr{N},{K_\text{PY}}} \right) =\left\{ {\left( {s,t,s} \right):s,t > 0} \right\}.\]
\end{example}

\section{Absolute concentration robustness in Hill-type kinetic systems}
\label{abs:con:rob:htk}
In this section, we apply the poly-PL method to study absolute concentration robustness (ACR) in HTK systems. We first provide a brief review of ACR concepts and related results for mass action and power-law kinetic systems. We then introduce the key concept of a Shinar-Feinberg (SF) pair of reactions for an HTK system via its associated PYK system and derive the Shinar-Feinberg ACR Theorem for a subset of deficiency one HT-RDK systems, which we call PL-equilibrated. This foundation enables the extension of ACR theory to deficiency zero HTK systems, and via network decomposition theory to larger and higher deficiency HTK systems.

\subsection{A brief review of ACR concepts and results for power-law kinetic systems}
The concept of {\bf absolute concentration robustness} ({\bf ACR}) was first introduced by Shinar and Feinberg in their well-cited paper published in Science \cite{SF2010}. ACR pertains to a phenomenon in which a species in a chemical kinetic system carries the same value for any positive steady state the network may admit regardless of initial conditions. In particular, a PL-RDK system $(\mathscr{N}, K)$ has ACR in a species $X \in \mathscr{S}$ if there exists $c^* \in E_+(\mathscr{N}, K)$ and for every other $c^{**} \in E_+(\mathscr{N}, K)$, we have $c^{**}_X=c^*_X$.

\begin{theorem} (Shinar-Feinberg Theorem on ACR for PL-RDK systems \cite {FLRM2020})
	\label{SFTACR}
	Let $\mathscr{N}=(\mathscr{S}, \mathscr{C},\mathscr{R})$ be a deficiency-one CRN and suppose that $(\mathscr{N}, K)$ is a PL-RDK system which admits a positive equilibrium. If $y, y' \in \mathscr{C}$ are nonterminal complexes whose kinetic order vectors differ only in species $X$, then the system has ACR in $X$.
\end{theorem}

Fortun and Mendoza \cite{FM2020} further investigated ACR in power-law kinetic systems and derived novel results that guarantee ACR for some classes of PLK systems. For these PLK systems, the key property for ACR in a species $X$ is the presence of an SF-reaction pair, which is defined as follows.

\begin{definition}
	A pair of reactions in a PLK system is called a {\bf Shinar-Feinberg pair} (or {\bf SF-pair}) in a species $X$ if their kinetic order vectors differ only in $X$. A subnetwork of the PLK system is of {\bf SF-type} if it contains an SF-pair in $X$.
\end{definition}

\subsection{The Shinar-Feinberg ACR Theorem for PL-equilibrated HT-RDK systems}
A basic observation about absolute concentration robustness is expressed in the following Lemma:

\begin{lemma}
	If $K$ and $K'$ are kinetics on a chemical reaction network $\mathscr{N}$ and $$E_+(\mathscr{N}, K)=E_+(\mathscr{N}, K'),$$ then $K$ has ACR in a species $X$ if and only if $K'$ has ACR.
\end{lemma}

This observation enables us to use the associated poly-PL system to transfer ACR concepts and results to HTK systems. We begin with the key concept of an SF-pair:

\begin{definition}
	A pair of reactions $R$ and $R'$ in an HTK system $(\mathscr{N}, K)$ is an SF-pair if its associated PYK system has the following property:
	if there is at least one $K_j$ such that the rows differ only in $X$.
	%\begin{itemize}
	%\item[i.] the rows for $R$ and $R'$ of the kinetic order matrix of $K_j$ differ at most in $X$ for all $j$,
	%\item[ii.] there is at least one $K_j$ where the rows differ in X, and
	%\item[iii.] ${A_{R,j}} = {A_{R',j}}$ for all $j \in \{ 1,...,h \}$.
	%\end{itemize}
\end{definition}

\begin{figure}
	\begin{center}
		\includegraphics[width=10cm,height=3cm,keepaspectratio]{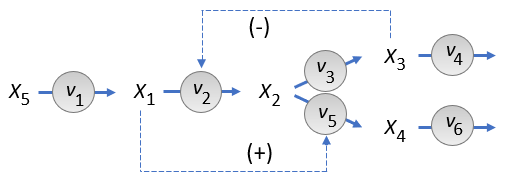}
		\caption{Biochemical map of a metabolic network with one positive feedforward and a negative feedback \cite{SHVA2007}}
		\label{sorribas:metabolic}
	\end{center}
\end{figure}

\begin{example}
	We now consider a particular biological system, illustrated in Figure \ref{sorribas:metabolic}, which is a metabolic network with one positive feedforward and a negative feedback obtained from the published work of Sorribas et al. \cite{SHVA2007}. The following CRN corresponds to the metabolic network with $X_5$ as independent variable.
	
	\[ \begin{array}{lll}
		R_1: 0 \to X_1 & \ \ \ &  R_4: X_1+X_2  \to X_1+X_4\\
		R_2: X_1+X_3  \to  X_3+X_2  &  \ \ \ & R_5: X_3 \to 0 \\
		R_3: X_2  \to X_3 &  \ \ \ & R_6: X_4 \to 0\\
	\end{array}\]
	In addition, the following are the kinetic order and dissociation constant matrices:
	\[F = \left[ {\begin{array}{*{20}{c}}
			0&0&0&0\\
			{{n_{21}}}&0&{{n_{23}}}&0\\
			0&{{n_{32}}}&0&0\\
			{{n_{41}}}&{{n_{42}}}&0&0\\
			0&0&{{n_{53}}}&0\\
			0&0&0&{{n_{64}}}
	\end{array}} \right] \ \ \ D = \left[ {\begin{array}{*{20}{c}}
			0&0&0&0\\
			{{k_{21}}}&0&{{k_{23}}}&0\\
			0&{{k_{32}}}&0&0\\
			{{k_{41}}}&{{k_{42}}}&0&0\\
			0&0&{{k_{53}}}&0\\
			0&0&0&{{k_{64}}}
	\end{array}} \right]\]
	with $n_{21} = 1$, $n_{23} = -0.8429$, $n_{32} = 1$, $n_{41} = 2.9460$, $n_{42} = 3$, $n_{53} = 1$, $n_{64} = 1$, $k_{21} = 0.6705$, $k_{41} = 0.8581$, $k_{42} = 44.7121$, $k_{53} = 1$, and $k_{64} = 1$.
	Hence, the corresponding ODEs are given by the following.
	\begin{align*}
		\frac{{d{X_1}}}{{dt}} &= {V_1} - \frac{{{V_2}X_1^{{n_{21}}}X_3^{{n_{23}}}}}{{\left( {{k_{12}} + X_1^{{n_{21}}}} \right)\left( {{k_{23}} + X_3^{{n_{23}}}} \right)}}\\
		\frac{{d{X_2}}}{{dt}} &= \frac{{{V_2}X_1^{{n_{21}}}X_3^{{n_{23}}}}}{{\left( {{k_{12}} + X_1^{{n_{21}}}} \right)\left( {{k_{23}} + X_3^{{n_{23}}}} \right)}} - \frac{{{V_3}X_2^{{n_{32}}}}}{{{k_{32}} + X_2^{{n_{32}}}}} - \frac{{{V_4}X_1^{{n_{41}}}X_2^{{n_{42}}}}}{{\left( {{k_{41}} + X_1^{{n_{41}}}} \right)\left( {{k_{42}} + X_2^{{n_{42}}}} \right)}}\\
		\frac{{d{X_3}}}{{dt}} &= \frac{{{V_3}X_2^{{n_{32}}}}}{{{k_{32}} + X_2^{{n_{32}}}}} - \frac{{{V_5}X_3^{{n_{53}}}}}{{{k_{53}} + X_3^{{n_{53}}}}}\\
		\frac{{d{X_4}}}{{dt}} &= \frac{{{V_4}X_1^{{n_{41}}}X_2^{{n_{42}}}}}{{\left( {{k_{41}} + X_1^{{n_{41}}}} \right)\left( {{k_{42}} + X_2^{{n_{42}}}} \right)}} - \frac{{{V_6}X_4^{{n_{64}}}}}{{{k_{64}} + X_4^{{n_{64}}}}}
	\end{align*}
	One can check that $\{R_1,R_3\}$ in $\left(\mathscr{N},K_\text{PY}\right)$ is an SF-pair in species $X_2$, and hence, in the original system $\left(\mathscr{N},K_\text{PY}\right)$ too.
\end{example}

The following Definition identifies a subset of HTK systems with interesting ACR properties:

\begin{definition}
	An HTK system $(\mathscr{N}, K)$ is {\bf PL-equilibrated} if 
	$E_+(\mathscr{N}, K) = \cap E_+(\mathscr{N}, K_{\text{PY},j})$. The sets of PL-equilibrated HTK and HT-RDK systems are denoted by HTK$_{\text {PLE}}$ and HT-RDK$_{\text {PLE}}$ respectively.
\end{definition}

The foundation for our ACR theory is the following Theorem:

\begin{theorem}
	(Shinar-Feinberg ACR Theorem for HT-RDK$_{\text {PLE}}$). Let $(\mathscr{N}, K)$ be a deficiency one PL-equilibrated HT-RDK system. If $E_+(\mathscr{N}, K) \ne \varnothing$ and an SF-pair in a species $X$ exists, then $(\mathscr{N}, K)$ has ACR in $X$.
	\label{shinar:feinberg:theorem:htdrk}
\end{theorem}

\begin{proof}
	Let $K_{\text {PY}} = K_{\text {PY}, 1} + ... + K_{\text {PY}, h}$  be the canonical PL-representation of the associated PY-RDK system.
	Note that $(\mathscr{N}, K_{\text {PY}, j})$ is a deficiency one PL-RDK with $E_+(\mathscr{N}, K_{\text {PY}, j}) \ne \varnothing$.
	By definition of an SF-pair, there is at least one $j^*$ with an SF-pair in $X$. It follows from Theorem \ref{SFTACR} that $(\mathscr{N}, K_{\text {PY}, j^*})$ has ACR in $X$. Since $E_+(\mathscr{N}, K) \subseteq E_+(\mathscr{N}, K_{\text {PY}, j^*})$, the system $(\mathscr{N}, K)$ has ACR in $X$. 
\end{proof}

\subsection{ACR in deficiency zero Hill-type kinetic systems}
\label{ACR:def:zero:hill:type}
We begin this section again with another basic observation:
\begin{lemma}
	If two kinetic systems $(\mathscr{N}, K)$ and $(\mathscr{N}', K')$ with the same species set are dynamically equivalent, then $K$ has ACR in $X$ if and only if $K'$ has ACR in $X$.
\end{lemma}
The claim follows from the fact that dynamic equivalence implies that the equilibria sets coincide.

In the following, we derive a general result about CF kinetics and a subset of NF kinetics and formulate the result for HTK as a Corollary. We recall the concept of CF interaction set from \cite{NEML2019}:

For a reactant complex $y$ of a network $\mathscr{N}$, $\mathscr{R}(y)$ denotes its set of (branching) reactions, i.e., $\rho ^{-1}(y)$ where $\rho:\mathscr{R} \to \mathscr{C}$ is the reactant map. The $n_r$ reaction sets $\mathscr{R}(y)$ of reactant complexes partition the set of reactions $\mathscr{R}$ and hence induce a decomposition of $\mathscr{N}$.

\begin{definition}
	Two reactions $r,r' \in \mathscr{R}(y)$ are {\bf{CF-equivalent}} for $K$ is their interaction functions coincide, i.e., ${I_{K,r}} = {I_{K,r'}}$ or, equivalently, if their kinetic functions $K_r$ and $K_{r'}$ are proportional (by a positive constant). The equivalence classes are the {\bf{CF-subsets}} (for $K$) of the reactant complex $y$.
\end{definition}

Let $N_R(y)$ be the number of CF-subsets of $y$. Then $1 \le {N_R}\left( y \right) \le \left| {{\rho ^{ - 1}}\left( y \right)} \right|$.

\begin{definition}
	The reactant complex is a {\bf{CF-node}} if ${N_R}\left( y \right)=1$, and an {\bf{NF-node}} otherwise. It is a {\bf{maximally NF-node}} if $ {N_R}\left( y \right) = \left| {{\rho ^{ - 1}}\left( y \right)} \right| >1$.
\end{definition}

\begin{definition}
	The number $N_R$ of CF-subsets of a CRN is the sum of ${N_R}\left( y \right)$ over all reactant complexes.
\end{definition}

Clearly, ${N_R} \ge {n_r}$ and the kinetics $K$ is CF if and only if ${N_R} = {n_r}$, or equivalently all reactant complexes are CF-nodes for $K$.

For our considerations, the following Definition is key:
\begin{definition}
	An NF kinetic system is {\bf{minimally NFK}} if it contains a single NF node $y$ and for that node, $N_R(y) = 2$ and at least one CF-subset has only one element. 
\end{definition}

The CF-RM$_+$ method, an algorithm introduced in \cite{NEML2019}, transforms an NFK system to a dynamically equivalent CFK. The procedure is as follows: at each NDK node, except for a CF-subset with a maximal number of reactions, the reactions in a CF-subset are replaced by adding the same reactant multiple to a reactant and product complexes, such that the new reactants and products do not coincide with any existing complexes. Suppose $(\mathscr{N}, K)$ is an NFK system that is transformed into a CFK system $(\mathscr{N}^*, K^*)$ via the CF-RM$_+$ method. The two key properties of $\mathscr{N}$ and $\mathscr{N}^*$ are the invariance of the stoichiometric subspaces, i.e., $S=S^*$,
and the kinetic functions, with only some of their indices, i.e., the transformed reactions, being changed. In particular, since the kinetic functions remains the same, $(\mathscr{N}, K)$ and $(\mathscr{N}^*, K^*)$ belong to the same class of kinetic systems. For instance, if $(\mathscr{N}, K)$ is an HTK system, then $(\mathscr{N}^*, K^*)$ is also an HTK system. 
%The details of the algorithm can be found in \cite{NEML2019}.

\begin{theorem}
	Let $\mathscr{M}$ be a set of RID kinetics for which the Shinar-Feinberg ACR Theorem holds and $K \in \mathscr{M}$. If $(\mathscr{N}, K)$ is a deficiency zero CFK or a minimally NFK system with a positive equilibrium and an SF-pair in $X$, then $(\mathscr{N}, K)$ has ACR in $X$.
\end{theorem}

\begin{proof}
	After the classical results of Feinberg and Horn, any positive equilibrium in a deficiency zero system is complex balanced \cite{feinberg:complex:balancing}. Thus, the underlying CRN is weakly reversible \cite{Horn:ns}.
	
	For the NFK case, let $y \to y'$ be a single reaction with the hypothesized CF-subset of the minimal NFK node. Applying the CF-RM$_+$ method to transform $(\mathscr{N}, K)$, we obtain a transform of $\mathscr{N}$, which is $\mathscr{N}^*$ with $\mathscr{S}^*=\mathscr{S}$, $\mathscr{C}^*=\mathscr{C} \cup \{y+ay,y'+ay\}$ where $a$ is an appropriate integral multiple of $y$ and $\mathscr{R}^*=\mathscr{R} \cup \{y+ay \to y'+ay\}$. Since we assume that the network has a complex balanced equilibrium, then by a classical result of Horn \cite{Horn:ns}, it is weakly reversible. Hence, each linkage class is weakly reversible. Since each reaction in the linkage class of the NFK node is in a cycle, the CR-RM$_+$ creates only one additional linkage class, namely $\{y+ay \to y'+ay\}$. Thus, the deficiency of $\mathscr{N}^*$ is $\delta ^* = (n+2)-(\ell +1) -s = \delta + 1 = 1$. The kinetic order matrix remains the same, so the transform $\mathscr{N}^*$ is still of SF-type in $X$. $\mathscr{N}^*$, as a dynamically equivalent CFK system, has a positive equilibrium. Hence, it fulfills the assumptions of the extension of the Shinar-Feinberg ACR Theorem for $\mathscr{M}$. Therefore, $\mathscr{N}^*$ has ACR in $X$.
	
	For the CFK case, we can apply the CR-RM$_+$ method to any reaction and also obtain an appropriate dynamically equivalent deficiency one system as in minimally NFK case.
\end{proof}

\begin{corollary}
	Let $K \in \rm{HTK_{PLE}}$. If $(\mathscr{N}, K)$ is a deficiency zero HT-RDK or a minimally NDK HTK system with a positive equilibrium and an SF-pair in $X$, then $(\mathscr{N}, K)$ has ACR in $X$.
\end{corollary}
\begin{proof}
	We showed in the previous section that the Shinar-Feinberg ACR Theorem holds for $\rm{HTK_{PLE}}$.
\end{proof}

\subsection{ACR in larger HTK systems}
Having established sufficient conditions for ACR in networks of low deficiency, i.e., $\delta = 0$ or 1, we are now in a position to use general results in \cite{FTJM2020} to turn such systems into ``ACR building blocks'' for larger and higher deficiency systems. We recall their result:
%(\cite{FTJM2020}, Theorem ):
\begin{theorem}
	Let $(\mathscr{N}, K)$ be an RIDK system with a positive equilibrium and an independent decomposition $\mathscr{N} =\mathscr{N}_1 \cup \mathscr{N}_2 \cup ... \cup \mathscr{N}_p$. Suppose there is an $\mathscr{N}_i$ with $(\mathscr{N}_i, K_i)$ of SF-type in $X \in \mathscr{S}$ such that
	\begin{itemize}
		\item[i.] $\delta _i =0$ with CF or minimally NF kinetics; or
		\item[ii.] $\delta _i =1$ with CF kinetics.
	\end{itemize}
	Then $(\mathscr{N}, K)$ has ACR in $X$.
\end{theorem}

The preceding Theorem applies to any $\rm{HTK_{PLE}}$ system satisfying the assumptions about a positive equilibrium and an independent decomposition.

\subsection{A comparison with ACR in the associated PLK system}
Since the numerator of a Hill-type kinetics is in power-law form, one may try to relate ACR in both systems. We first give the numerator-induced system a formal name:

\begin{definition}
	The associated PLK system $(\mathscr{N}, K_{\text{PL}})$ of an HTK system $(\mathscr{N}, K)$ is given, for each reaction $q$, by $K_{\text{PL},q} (x) = k_qM_q(x)$.
\end{definition}
Recall that $M_q(x) := x^{F_q}$, where $F_q$ is the $q$-th row of the kinetic order matrix $F$.
%In the following statement, SF-pair on the LHS is according to Definition 18, while on the RHS, it is according to Definition 17,

\begin{table}
	\begin{tabular}{ c c }
		\toprule
		%\textbf{EKF} & \textbf{UKF} & ABC & \\
		%\midrule\\
		%\addlinespace[-2ex]
		\textbf{Example a}&
		$F = \left[ {\begin{array}{*{20}{c}}
				1&1\\
				0&0
		\end{array}} \right]$ 
		$D = \left[ {\begin{array}{*{20}{c}}
				1&0\\
				0&1
		\end{array}} \right]$ 
		$K\left( x \right) = \left[ {\begin{array}{*{20}{c}}
				{{k_1}\dfrac{{{x_1}{x_2}}}{{\left( {1 + {x_1}} \right){x_2}}}}\\
				{{k_2}/2}
		\end{array}} \right]$
		\\
		\addlinespace[1.5ex]
		& ${K_{\text{PY}}}\left( x \right) = \left[ {\begin{array}{*{20}{c}}
				{{k_1}\left( {{x_1}{x_2} + {x_1}{x_2}} \right)}\\
				{{k_2}\left( {{x_2} + {x_1}{x_2}} \right)}
		\end{array}} \right]$ 
		${F_1} = \left[ {\begin{array}{*{20}{c}}
				1&1\\
				0&1
		\end{array}} \right]$ 
		${F_2} = \left[ {\begin{array}{*{20}{c}}
				1&1\\
				1&1
		\end{array}} \right]$ \\
		\midrule
		\textbf{Example b}&
		$F = \left[ {\begin{array}{*{20}{c}}
				1&1\\
				0&0
		\end{array}} \right]$
		$D = \left[ {\begin{array}{*{20}{c}}
				1&0\\
				0&0
		\end{array}} \right]$
		$K\left( x \right) = \left[ {\begin{array}{*{20}{c}}
				2{{k_1}\dfrac{{{x_1}{x_2}}}{{\left( {1 + {x_1}} \right){x_2}}}}\\
				{{k_2}}
		\end{array}} \right]$\\
		&
		${K_{\text{PY}}}\left( x \right) = \left[ {\begin{array}{*{20}{c}}
				{{k_1}\left( {{x_1}{x_{_2}} + {x_1}{x_2}} \right)}\\
				{{k_2}\left( {{x_2} + {x_1}{x_2}} \right)}
		\end{array}} \right]$
		${F_1} = \left[ {\begin{array}{*{20}{c}}
				1&1\\
				0&1
		\end{array}} \right]$ ${F_2} = \left[ {\begin{array}{*{20}{c}}
				1&1\\
				1&1
		\end{array}} \right]$
		\\
		\midrule
		\textbf{Example c}&
		$F = \left[ {\begin{array}{*{20}{c}}
				1&0\\
				0&1
		\end{array}} \right]$ $D = \left[ {\begin{array}{*{20}{c}}
				0&0\\
				0&1
		\end{array}} \right]$
		$K\left( x \right) = \left[ {\begin{array}{*{20}{c}}
				{{k_1}\dfrac{{{x_1}}}{{{x_1}}}}\\
				2{{k_2}\dfrac{{{x_2}}}{{1 + {x_2}}}}
		\end{array}} \right]$\\
		& ${K_{\text{PY}}}\left( x \right) = \left[ {\begin{array}{*{20}{c}}
				{{k_1}\left( {{x_1} + {x_1}{x_2}} \right)}\\
				{{k_2}\left( {{x_1}{x_{_2}} + {x_1}{x_{_2}}} \right)}
		\end{array}} \right]$
		${F_1} = \left[ {\begin{array}{*{20}{c}}
				1&0\\
				1&1
		\end{array}} \right]$ ${F_2} = \left[ {\begin{array}{*{20}{c}}
				1&1\\
				1&1
		\end{array}} \right]$
		\\
		\midrule
		\textbf{Example d}&
		$F = \left[ {\begin{array}{*{20}{c}}
				1&1\\
				1&1
		\end{array}} \right]$ $D = \left[ {\begin{array}{*{20}{c}}
				1&0\\
				0&0
		\end{array}} \right]$
		$K\left( x \right) = \left[ {\begin{array}{*{20}{c}}
				2{{k_1}\dfrac{{{x_1}{x_2}}}{{\left( {1 + {x_1}} \right){x_2}}}}\\
				{{k_2}\dfrac{{{x_1}{x_2}}}{{{x_1}{x_2}}}}
		\end{array}} \right]$\\
		& 
		${K_{\text{PY}}}\left( x \right) = \left[ {\begin{array}{*{20}{c}}
				{{k_1}\left( {{x_1}^2{x_2}^2 + {x_1}^2{x_2}^2} \right)}\\
				{{k_2}\left( {{x_1}{x_2}^2 + {x_1}^2{x_2}^2} \right)}
		\end{array}} \right]$
		${F_1} = \left[ {\begin{array}{*{20}{c}}
				2&2\\
				1&2
		\end{array}} \right]$ ${F_2} = \left[ {\begin{array}{*{20}{c}}
				2&2\\
				2&2
		\end{array}} \right]$
		\\
		\bottomrule\\
	\end{tabular}
	\caption{Cases where systems $(\mathscr{N}, K)$ with given kinetics have an SF-pair in $X_i$ but their corresponding $(\mathscr{N}, K_{\text{PL}})$ have no SF-pair in $X_i$}
	\label{table:hilltype:plk1}
\end{table}

For Table \ref{table:hilltype:plk1}, we let $(\mathscr{N}, K)$ be an HTK system with given kinetic order matrix $F$ and dissociation constant matrix $D$.
The systems $(\mathscr{N}, K)$ with kinetics depicted in Examples a, b and d have an SF-pair in $X_1$ but $(\mathscr{N}, K_{\text{PL}})$ have no SF-pair in $X_1$. In addition, the system $(\mathscr{N}, K)$ with kinetics depicted in Example c has an SF-pair in $X_2$ but $(\mathscr{N}, K_{\text{PL}})$ has no SF-pair in $X_2$.

\begin{table}
	\begin{tabular}{ c c }
		\toprule
		%\textbf{EKF} & \textbf{UKF} & ABC & \\
		%\midrule\\
		%\addlinespace[-2ex]
		\textbf{Example e}&
		$F = \left[ {\begin{array}{*{20}{c}}
				1&1\\
				0&1
		\end{array}} \right]$ $D = \left[ {\begin{array}{*{20}{c}}
				1&0\\
				0&1
		\end{array}} \right]$
		$K\left( x \right) = \left[ {\begin{array}{*{20}{c}}
				{{k_1}\dfrac{{{x_1}{x_2}}}{{\left( {1 + {x_1}} \right){x_2}}}}\\
				{{k_2}\dfrac{{{x_2}}}{{1 + {x_2}}}}\\
		\end{array}} \right]$\\
		&
		${K_{\text{PY}}}\left( x \right) = \left[ {\begin{array}{*{20}{c}}
				{{k_1}\left( {{x_1}{x_2} + {x_1}{x_2}^2} \right)}\\
				{{k_2}\left( {{x_2}^2 + {x_1}{x_2}^2} \right)}
		\end{array}} \right]$
		${F_1} = \left[ {\begin{array}{*{20}{c}}
				1&1\\
				0&2
		\end{array}} \right],{F_2} = \left[ {\begin{array}{*{20}{c}}
				1&2\\
				1&2
		\end{array}} \right]$
		\\
		\midrule
		\textbf{Example f}&
		$F = \left[ {\begin{array}{*{20}{c}}
				1&1\\
				0&1
		\end{array}} \right]$ $D = \left[ {\begin{array}{*{20}{c}}
				1&1\\
				0&1
		\end{array}} \right]$
		$K\left( x \right) = \left[ {\begin{array}{*{20}{c}}
				{{k_1}\dfrac{{{x_1}{x_2}}}{{\left( {1 + {x_1}} \right)\left( {1 + {x_2}} \right)}}}\\
				{{k_2}\dfrac{{{x_2}}}{{1 + {x_2}}}}
		\end{array}} \right]$
		\\
		&
		${K_{\text{PY}}}\left( x \right) = \left[ {\begin{array}{*{20}{c}}
				{{k_1}\left( {{x_1}{x_2} + {x_1}{x_2}^2} \right)}\\
				{{k_2}\left( {{x_2} + {x_2}^2 + {x_1}{x_2} + {x_1}{x_2}^2} \right)}
		\end{array}} \right]$\\
		&
		${F_1} = \left[ {\begin{array}{*{20}{c}}
				1&1\\
				0&1
		\end{array}} \right]$ ${F_2} = \left[ {\begin{array}{*{20}{c}}
				1&2\\
				1&1
		\end{array}} \right]$ ${F_3} = \left[ {\begin{array}{*{20}{c}}
				1&2\\
				0&2
		\end{array}} \right]$ ${F_4} = \left[ {\begin{array}{*{20}{c}}
				1&2\\
				1&2
		\end{array}} \right]$\\
		\midrule
		\textbf{Example g}
		&
		$F = \left[ {\begin{array}{*{20}{c}}
				1&0\\
				1&1
		\end{array}} \right]$ $D = \left[ {\begin{array}{*{20}{c}}
				1&0\\
				0&1
		\end{array}} \right]$
		$K\left( x \right) = \left[ {\begin{array}{*{20}{c}}
				{{k_1}\dfrac{{{x_1}}}{{\left( {1 + {x_1}} \right)}}}\\
				{{k_2}\dfrac{{{x_1}{x_2}}}{{{x_1}\left( {1 + {x_2}} \right)}}}
		\end{array}} \right]$
		\\
		&
		${K_{\text{PY}}}\left( x \right) = \left[ {\begin{array}{*{20}{c}}
				{{k_1}\left( {{x_1}^2 + {x_1}^2{x_2}} \right)}\\
				{{k_2}\left( {{x_1}{x_2} + {x_1}^2{x_2}} \right)}
		\end{array}} \right]$
		${F_1} = \left[ {\begin{array}{*{20}{c}}
				2&0\\
				1&1
		\end{array}} \right]$ ${F_2} = \left[ {\begin{array}{*{20}{c}}
				2&1\\
				2&1
		\end{array}} \right]$
		\\
		\midrule
		\textbf{Example h}
		&
		$F = \left[ {\begin{array}{*{20}{c}}
				1&0\\
				1&1
		\end{array}} \right]$ $D = \left[ {\begin{array}{*{20}{c}}
				1&1\\
				0&1
		\end{array}} \right]$
		$K\left( x \right) = \left[ {\begin{array}{*{20}{c}}
				{{k_1}\dfrac{{{x_1}}}{{\left( {1 + {x_1}} \right)\left( 2 \right)}}}\\
				{{k_2}\dfrac{{{x_1}{x_2}}}{{{x_1}\left( {1 + {x_2}} \right)}}}
		\end{array}} \right]$
		\\
		&
		${K_{\text{PY}}}\left( x \right) = \left[ {\begin{array}{*{20}{c}}
				{{k_1}\left( {{x_1}^2 + {x_1}^2{x_2}} \right)}\\
				{2{k_2}\left( {{x_1}{x_2} + {x_1}^2{x_2}} \right)}
		\end{array}} \right]$
		${F_1} = \left[ {\begin{array}{*{20}{c}}
				2&0\\
				1&1
		\end{array}} \right]$ ${F_2} = \left[ {\begin{array}{*{20}{c}}
				2&1\\
				2&1
		\end{array}} \right]$
		\\
		\bottomrule\\
	\end{tabular}
	\caption{Cases where systems $(\mathscr{N}, K_{\text{PL}})$ with given kinetics have an SF-pair in $X_i$ but their corresponding $(\mathscr{N}, K)$ have no SF-pair in $X_i$}
	\label{table:hilltype:plk2}
\end{table}

We now refer to Table \ref{table:hilltype:plk2}. We again let $(\mathscr{N}, K)$ be an HTK system.
The systems $(\mathscr{N}, K_{\text{PL}})$ with kinetics depicted in Examples e and f have an SF-pair in $X_1$ but $(\mathscr{N}, K)$ have no SF-pair in $X_1$. Moreover, the systems $(\mathscr{N}, K_{\text{PL}})$ with kinetics depicted in Examples c and d have an SF-pair in $X_2$ but $(\mathscr{N}, K)$ have no SF-pair in $X_2$.

\section{Complex balanced equilibria of weakly reversible Hill-type kinetic systems}
\label{com:bal:eq:HTK}
In this section, we focus on complex balanced equilibria of weakly reversible HTK systems. We first extend concepts and partial results of M\"uller and Regensburger on generalized mass action systems to complex balanced HT-RDK systems.
The results of Talabis et al. \cite{TMMN2020} also hold for some HT-RDK systems.
This leads to the identification of the subset of PL-complex balanced systems for which the extended results are complete. We then present examples of PL-complex balanced systems. In analogy to the ACR theory for PL-equilibrated systems, we develop the theory of balanced concentration robustness (BCR) for PL-complex balanced systems. 

\subsection{Complex balancing of Hill-type kinetics}
We first recall the two kinds of complex balancing properties of an RID kinetics on a given network:
\begin{definition}
	A kinetics $K$ on a weakly reversible network $\mathscr{N}$ has conditional complex balancing (CCB) if there exist rate constants such that $(\mathscr{N}, K)$ is complex balanced, i.e., $Z_+(\mathscr{N}, K)\ne \varnothing$.  A kinetics is unconditionally complex balanced (UCB) on a weakly reversible network if for all rate constants, $Z_+(\mathscr{N}, K)\ne \varnothing$.
\end{definition}
We first observe that conditional complex balancing holds for any weakly reversible HT-RDK system:

\begin{theorem}
	For any $K \in \rm{HT}\text{-}\rm{RDK}(\mathscr{N})$ on a weakly reversible network $\mathscr{N}$, there are rate constants such that  $Z_+(\mathscr{N}, K)\ne \varnothing$, i.e., it is conditionally complex balancing.
\end{theorem}

This result is a Corollary of Theorem 3 in \cite{FTJM2020} where it is shown that conditional complex balancing holds for $\rm{CFK}(\mathscr{N})$ on weakly reversible networks and by definition, $ \rm{HT}\text{-}\rm{RDK}(\mathscr{N}) =  \rm{CFK}(\mathscr{N})  \cap \rm{HTK}(\mathscr{N})$ .

To obtain results on unconditional complex balancing for $K \in \rm{HT}\text{-}\rm{RDK}$, the associated poly-PL kinetics $K_{\rm{PY}}$ is a key tool. We already observed in Proposition \ref{prop:KPY} that $K \in {\rm{HT}{\text{-}\rm{RDK}}} \Leftrightarrow {K_{\rm{PY}}} \in {\rm{PY}{\text{-}\rm{RDK}}}$. We can therefore use the extensions of the theory of M\"uller and Regensburger to PY-RDK in \cite{FTJM2020} for HT-RDK. This includes the fundamental concept of kinetic deficiency and the following results on UCB:
\begin{definition}
	The kinetic deficiency $\widetilde{\delta}$ of $K$ is the kinetic deficiency of its associated poly-PL kinetics.
\end{definition}
\begin{theorem}
	If the kinetic deficiency of $K$ on a network $\mathscr{N}$ is zero, then, for any set of rate constants, the kinetics is unconditionally complex balancing on $\mathscr{N}$.
\end{theorem} 
\begin{proof}
	This follows from Proposition 10 in \cite{FTJM2020} and Theorem \ref{theorem:equilibriasets}.
\end{proof}
In \cite{TMMN2020} the concept of kinetic reactant deficiency $\widehat{\delta}$ was extended to poly-PL systems. By defining it analogously for HT-RDK via the associated poly-PL kinetics, we obtain the following Corollary:
\begin{corollary}
	If the kinetic reactant deficiency of a  $K \in \rm{HT}\text{-}\rm{RDK}$ on a weakly reversible network is zero, then it is unconditionally complex balancing on $\mathscr{N}$. 
\end{corollary}
\begin{proof}
	Let $K_\text{PY} = K_{\text{PY},1} + ... + K_{\text{PY},h}$ be the canonical PL-representation of $K_\text{PY}$. Zero kinetic reactant deficiency of $K$ implies that the kinetic reactant deficiency of each of the PL-kinetics $K_{\text{PY},j}$ is zero. In view of Theorem 1 in \cite{MURE2014}, it follows that the kinetic deficiency is zero for each $K_{\text{PY},j}$, which in turn implies that $K$ has zero kinetic deficiency and hence UCB.
\end{proof}

\subsection{Equilibria parametrization and multistationarity in Hill-type systems}
In this section, we formally define the equilibria subsets to which the parametrization and multistationarity results of M\"uller and Regensburger for GMAS can be extended.

\begin{definition}
	The subsets {\bf{PL-equilibria}} and {\bf{PL-complex balanced equilibria}} of an HTK system are defined as
	$E_{+,\text{PL}}(\mathscr{N}, K) := E_{+,\text{PL}}(\mathscr{N}, K_\text{PY})$ and $Z_{+,\text{PL}}(\mathscr{N}, K) := Z_{+,\text{PL}}(\mathscr{N}, K_\text{PY})$, respectively.
\end{definition}

%The associated PYK system is uniquely derived from the definition of an HTK system. Hence, it can be used to define concepts for the HTK system which are currently not available there and demonstrate analogous results. Our first result is the following.

%\begin{proposition}
%\label{prop:conditional:complex:balancing}
%{\bf (Conditional Complex Balancing)}. Let $(\mathscr{N}, K)$ be a weakly reversible HT-RDK kinetic system. Then there are rate constants such that the system is complex balanced.
%\end{proposition}
%\begin{proof}
%It follows from the Proposition \ref{prop:hilltype:assockinetics} that $K_\text{PY} \in$ PY-RDK. In \cite{FTJM2020}, it was shown that for any weakly reversible PY-RDK system, there are rate constants such that it is complex balanced, i.e., $Z_+(\mathscr{N}, K_\text{PY}) \ne \varnothing$. Since  $Z_+(\mathscr{N}, K) = Z_+(\mathscr{N}, K_\text{PY})$, then the claim follows.
%\end{proof}

%We use the associated poly-PL system also to define the concepts of kinetic order subspace $\widetilde S$, kinetic deficiency $\widetilde \delta = \dim \widetilde S$ and sign vector spaces. This allows us to state the corresponding results for HT-RDK systems as follows:
We can now state the extensions of parametrization and multistationarity:

\begin{theorem}
	Let $(\mathscr{N}, K)$ and $(\mathscr{N}, K_{\text{PY}})$ be a weakly reversible HT-RDK system and its associated PY-RDK system, respectively. Then, we have:
	\begin{itemize}
		\item[i.] {\bf Parametrization of PL-complex balanced equilibria set}: $$Z_{+,PL}(\mathscr{N}, K) = \left\{ {c \in {\mathbb{R}_{ > 0}^m}|\ln \left( c \right) - \ln \left( {{c^*}} \right) \in {{\left( \widetilde S \right)}^\perp}} \right\}$$
		\item[ii.] {\bf Multistationarity Criterion for PL-complex balanced equilibria}: A weakly reversible PY-RDK system has two distinct complex balanced PL-equilibria in a stoichiometric class $\Leftrightarrow$ ${\rm sign} (S) \cap {\rm sign} {\left( \widetilde S \right)}^\perp = 0$.
	\end{itemize}
	\label{mult:criterion:parametrization}
\end{theorem}

\begin{proof}
	These properties carry over from the associated poly-PL system derived in \cite{FTJM2020}. 
\end{proof}

\begin{remark}
	In \cite{WIFE2013}, the authors observed many similarities between power-law and Hill-type kinetics with respect to injectivity properties. Injective HTK systems form a subset of monostationary HTK systems since monostationarity is a necessary condition for injectivity. Monostationarity is clearly a question of equilibria sets, and in this regard, an HTK system $(\mathscr{N}, K)$ can be identified with its associated poly-PL system $(\mathscr{N}, K_\text{PY})$. This shows that the observed similarities derive from the fact that injective HTK systems can be embedded in the additive semigroup generated by PLK systems (where addition is coordinate-wise addition as defined in \cite{AJLM2017}).
\end{remark}

\subsection{PL-complex balanced HT-RDK systems}
We now introduce the analogue of PL-equilibrated systems for complex balanced equilibria.
\begin{definition}
	An HT-RDK system $(\mathscr{N}, K)$ is called {\bf{PL-complex balanced}} if it is complex balanced, i.e., $Z_+(\mathscr{N}, K) \ne \varnothing$, and 
	$Z_+(\mathscr{N}, K) = Z_{+,{\text{PL}}}(\mathscr{N}, K)$.
\end{definition}

We denote the set of PL-complex balanced HT-RDK systems with HT-RDK$_\text{PLC}$.

We can now state a fuller extension of the extensions in the previous two sections to this subset of HT-RDK systems:

\begin{theorem}
	If $K$ is a PL-complex balanced HT-RDK kinetics on a weakly reversible network, then the following statements hold:
	\begin{itemize}
		\item[i.] {\bf Unconditional Complex Balancing}: $\widetilde \delta =0$ if and only if $Z_+(\mathscr{N}, K) \ne \varnothing$ for any set of rate constants,
		\item[ii.] {\bf Parametrization}: $$Z_{+,PL}(\mathscr{N}, K) = \left\{ {c \in {\mathbb{R}_{ > 0}^m}|\ln \left( c \right) - \ln \left( {{c^*}} \right) \in {{\left( \widetilde S \right)}^\perp}} \right\},$$and
		\item[iii.] {\bf Multistationarity Criterion}: $(\mathscr{N}, K)$ is multistationary if and only if ${\rm sign} (S) \cap {\rm sign} {\left( \widetilde S \right)}^\perp = 0$.
	\end{itemize}
\end{theorem}

\begin{proof}
	Statements (ii) and (iii) follow from Theorem \ref{mult:criterion:parametrization} due to the assumed equality of sets of complex balanced equilibria and PL-complex balanced equilibria. To show (i), we only need to show ``$\Leftarrow$'': since for any set of rate constants,
	$Z_{+}(\mathscr{N}, K)$ is non-empty, the same is true for $Z_{+}(\mathscr{N}, K_{\text{PY}}) = Z_{+}(\mathscr{N}^*, K_{\text{PY}}^*)$, the latter being the STAR-MSC transform.
	This implies that each $(\mathscr{N}^*_j, K_{\text{PY},j}^*)$ has unconditional complex balancing, and hence after M\"uller and Regensburger, with kinetic deficiency zero. Thus, $(\mathscr{N}, K_{\text{PY}})$ and consequently $(\mathscr{N}, K)$ has zero kinetic deficiency.
\end{proof}

\subsection{Examples of PL-complex balanced HT-RDK systems}
We now provide some examples of HT-RDK systems which are PL-complex balanced. Note that both underlying networks are weakly reversible.
\begin{example}
	Let $(\mathscr{N}, K)$ defined by reactions ${R_1}:{X_1} \to {X_2}$ and ${R_2}:{X_2} \to {X_1}$ with $K\left( x \right) = \left[ {\begin{array}{*{20}{c}}
			{\dfrac{{{x_1}}}{{1 + {x_1}}}}\\
			{\dfrac{{{x_2}}}{{1 + {x_2}}}}
	\end{array}} \right]$
	$ \Rightarrow {I_a} = \left[ {\begin{array}{*{20}{c}}
			{ - 1}&1\\
			1&-1
	\end{array}} \right]$
	and
	${K_{{\text{PY}}}}\left( x \right) = \left[ {\begin{array}{*{20}{c}}
			{{x_1}\left( {1 + {x_2}} \right)}\\
			{{x_2}\left( {1 + {x_1}} \right)}
	\end{array}} \right].$
	We can easily show that ${Z_ + }\left( {\mathscr{N},{K}} \right) = \left\{ {t\left( {1,1} \right)|t > 0} \right\}\ne \varnothing.$
	Consider ${K_{{\rm{PY,1}}}}\left( x \right) = \left[ {\begin{array}{*{20}{c}}
			{{x_1}}\\
			{{x_2}}
	\end{array}} \right]$
	and
	${K_{{\rm{PY,2}}}}\left( x \right) = \left[ {\begin{array}{*{20}{c}}
			{{x_1}{x_2}}\\
			{{x_1}{x_2}}
	\end{array}} \right]$.
	Now, $${Z_ + }\left( {\mathscr{N},{K_{{\rm{PY,1}}}}} \right) = \left\{ {t\left( {1,1} \right)|t > 0} \right\}$$ and $${Z_ + }\left( {\mathscr{N},{K_{{\rm{PY,2}}}}} \right) = \left\{ {\left( {s,t} \right)|s,t > 0} \right\}.$$
	Thus, $${Z_ + }\left( {\mathscr{N},K} \right) = {Z_ + }\left( {\mathscr{N},{K_{{\rm{PY,1}}}}} \right) \cap {Z_ + }\left( {\mathscr{N},{K_{{\rm{PY,2}}}}} \right) = {Z_{ + ,{\rm{PL}}}}\left( {\mathscr{N},K} \right).$$
\end{example}

\begin{example}
	\label{example:running2}
	Referring to Example \ref{example:running1} in Section \ref{positive:equilibria:hill}, consider $(\mathscr{N}, K)$ with
	\[\begin{array}{*{20}{c}}
		{{R_1}:{X_1} \to {X_1} + {X_2}}\\
		{{R_2}:{X_1} + {X_2} \to {X_1}}\\
		{{R_3}:{X_3} \to {X_1}{\rm{ \ \ \ \ \ \ \  \  }}}
	\end{array}{\rm{ \ \   and \ \  }}K\left( x \right) = \left[ {\begin{array}{*{20}{c}}
			{\dfrac{{{x_1}}}{{1 + {x_1}}}}\\
			{\dfrac{{{x_1}}}{{1 + {x_1}}}}\\
			{\dfrac{{{x_3}}}{{1 + {x_3}}}}
	\end{array}} \right].\]
	Note that
	\[{K_{{\rm{PY}}}}\left( x \right) =
	\left[ {\begin{array}{*{20}{c}}
			{{x_1} + {x_1}{x_3}}\\
			{{x_1} + {x_1}{x_3}}\\
			{{x_3} + {x_1}{x_3}}
	\end{array}} \right].\]
	Also, let
	\[{K_{{\rm{PY,1}}}}\left( x \right) = \left[ {\begin{array}{*{20}{c}}
			{{x_1}}\\
			{{x_1}}\\
			{{x_3}}
	\end{array}} \right]{\rm{\ \ \ and \ \ \  }}{K_{{\rm{PY,2}}}}\left( x \right) = \left[ {\begin{array}{*{20}{c}}
			{{x_1}x_3}\\
			{{x_1}x_3}\\
			{{x_1}x_3}
	\end{array}} \right].\]
	We can verify that
	$${Z_ + }\left( {\mathscr{N},K} \right) = \bigcap\limits_{j = 1}^2  {Z_ + }\left( {\mathscr{N},{K_{{\rm{PY}},j}}} \right)  = {Z_{ + ,{\rm{PL}}}}\left( {\mathscr{N},K} \right).$$
\end{example}

\subsection{Balanced concentration robustness in PL-complex balanced systems}
In \cite{FM2020}, the following concept was introduced:
\begin{definition}
	A complex balanced chemical kinetic system $\left( {\mathscr{N},K} \right)$ has {\bf{balanced concentration robustness (BCR)}} in a species $X \in \mathscr{S}$ if $X$ has the same value for all $c \in Z_+\left( {\mathscr{N},K} \right)$.
\end{definition}

In analogy to Theorem \ref{shinar:feinberg:theorem:htdrk} for PL-equilibrated systems, we have the following result for PL-complex balanced systems:

\begin{theorem}
	(Shinar-Feinberg BCR Theorem for HT-RDK$_{\text{PLC}}$)
	Let $(\mathscr{N}, K)$ be a weakly reversible, deficiency one PL-complex balanced HT-RDK system.
	If $Z_+(\mathscr{N}, K) \ne \varnothing$ and an SF-pair in a species $X$ exists, then $(\mathscr{N}, K)$ has BCR in $X$.
\end{theorem}
The proof is easily adapted from that of Theorem \ref{shinar:feinberg:theorem:htdrk}. Since for deficiency zero networks, BCR coincides with ACR and in general, ACR implies BCR, the results of Section \ref{ACR:def:zero:hill:type} also hold for BCR. Again, using the general results of Section 6.4 in \cite{FTJM2020}, we obtain sufficient conditions for BCR in large and higher deficiency kinetic systems.

\section{Extension to poly-PL quotients and beyond}
\label{extension:beyond}
In this section, we outline extensions of the scope of the results presented for HTK systems. We first discuss the validity of the results for quotients of poly-PL kinetics and provide examples from enzyme kinetics as well as systems biology. We conclude the section with remarks on the extension to multiplicative monoids of kinetics and the groups they generate in the framework introduced by Arceo et al \cite{AJLM2017}.

\subsection{Extension to poly-PL quotient kinetics}
We presented the results so far in terms of Hill-type kinetics primarily because of their ubiquity and importance in biochemistry, in particular, in enzymology. However, we can observe that most of the results are more generally valid for kinetics, which when restricted to the positive orthant, are quotients of poly-PL kinetics, which we denote by $\text{PQK}(\mathscr{N})$.

For $K \in \text{PQK}(\mathscr{N})$ with $K_q = k_q \dfrac{{{M_q}\left( x \right)}}{{{T_q}\left( x \right)}}$, we denote by $T(x)$ the product $\prod\limits_{k = 1}^r {{T_k}\left( x \right)}$.
For a reaction $q$, we set $T_{(q)}(x) := {\prod\limits_{k \ne q} {{T_k(x)}} }$.
Since ${T_q}\left( x \right){T_{\left( q \right)}}\left( x \right) = T\left( x \right)$, we can write the interaction function of $K_q$, $\dfrac{{{M_q}\left( x \right)}}{{{T_q}\left( x \right)}} = \dfrac{{{M_q}\left( x \right){T_{\left( q \right)}}\left( x \right)}}{{T\left( x \right)}}$.

\begin{theorem}
	Let $K \in \text{PQK}(\mathscr{N})$ and $K_{\rm{PY}}$ its associated PYK given by $$K_{\text{PY},q}(x):=k_q M_q(x)T_{(q)}(x).$$ We have
	\begin{itemize}
		\item[i.] $Z_+(\mathscr{N}, K) = Z_+(\mathscr{N}, K_\text{PY})$, and
		\item[ii.] $E_+(\mathscr{N}, K) = E_+(\mathscr{N}, K_\text{PY})$.
	\end{itemize}
\end{theorem} 
\begin{proof}
	The proof is identical to Theorem \ref{theorem:equilibriasets}.
\end{proof}

\begin{example}
	Referring to the running example (i.e., Example 3.7) on pages 16 and 25 of \cite{MHRM2020}, consider the following CRN:
	\begin{align*}
		{R_1}&:5X + Y \to X + 3Y\\
		{R_2}&:X + 3Y \to 5X + Y
	\end{align*}
	but endowed with the following kinetics:
	$$K\left( X,Y \right) = \left[ {\begin{array}{*{20}{c}}
			k_1\left({\dfrac{{{\alpha _1}X + {\alpha _2}Y}}{X} + {\alpha _4}X + {\alpha _3}Y}\right)\\
			k_2\left({\dfrac{1}{X}\left( {\dfrac{{{\beta _2}}}{X} + \dfrac{{{\beta _1}}}{Y}} \right) + \dfrac{{{\beta _3}}}{X} + \dfrac{{{\beta _4}}}{Y}}\right)
	\end{array}} \right],$$
	which can be written in the following manner:
	\[K\left( {X,Y} \right) = \left[ {\begin{array}{*{20}{c}}
			k_1{\dfrac{{{\alpha _1}X + {\alpha _2}Y + {\alpha _3}XY + {\alpha _4}{X^2}}}{X}}\\
			k_2{\dfrac{{{\beta _1}X + {\beta _2}Y + {\beta _3}XY + {\beta _4}{X^2}}}{{{X^2}Y}}}
	\end{array}} \right].\]
	Since the LCM of the denominators is $X^2Y$, we obtain the kinetics $K_{\text{PY}}\left( X,Y \right)$ precisely in the running example of \cite{MHRM2020}:
	$$K_{\text{PY}}\left( X,Y \right)
	= \left[ {\begin{array}{*{20}{c}}
			k_1 \left(\alpha_1 X^2Y+  \alpha_2 XY^2 + \alpha_3 X^2Y^2 +\alpha_4 X^3Y \right)\\
			k_2 \left(\beta_1 X+  \beta_2 Y + \beta_3 XY + \beta_4 X^2 \right) 
	\end{array}} \right].$$
	Since $(\mathscr{N}, K_\text{PY})$ has the capacity for multistationarity for particular rate constants \cite{MHRM2020}, then so is the system $(\mathscr{N}, K)$.
\end{example}

\begin{example}
	Many kinetics for enzymatic reactions are examples for PQK. As documented in the Appendix 2 of \cite{GAHA2015}, among them are the following: ordered bi uni, allosteric inhibition (reversible), mixed inhibition (irreversible), catalytic activation (irreversible) and both irreversible variants of substrate activation and inhibition.
\end{example}

\begin{example}
	An example from Systems Biology is a model of {\it{Mycobacterium tuberculosis}} (Mtb) development within a human host by Magombedze and Mulder \cite{MAMU2012}. The system with 8 species has the following ODE system:\\
	$$\dfrac{{d{X_1}}}{{dt}} = {\alpha _E}\left( {\dfrac{{{X_5}}}{{{X_5} + {K_{MO}}}}} \right){X_1} - {\alpha _1}\left( {\dfrac{{{X_7}}}{{{X_7} + {K_{ML}}}}} \right){X_1} + {\gamma _1}\left( {1 - \dfrac{{{X_7}}}{{{X_7} + {K_{ML}}}}} \right){X_2}$$
	\hspace{1.1cm}$- {\mu _E}{X_1}$\\
	$\dfrac{{d{X_2}}}{{dt}} = {\alpha _L}\left( {\dfrac{{{X_5}}}{{{X_5} + {K_{MO}} + {F_a}}}} \right){X_1} + {\alpha _1}\left( {\dfrac{{{X_7}}}{{{X_7} + {K_{ML}}}}} \right){X_1} - {\alpha _2}\left( { \dfrac{{{X_8}}}{{{X_8} + {K_{MD}}}}} \right){X_2}$\\
	\[ - {\gamma _1}\left( {1 - \frac{{{X_7}}}{{{X_7} + {K_{ML}}}}} \right){X_2} + {\gamma _2}\left( {1 - \frac{{{X_8}}}{{{X_8} + {K_{MD}}}}} \right){X_3} - {\mu _L}{X_2}\]
	$\dfrac{{d{X_3}}}{{dt}} = {\alpha _2}\left( {\dfrac{{{X_8}}}{{{X_8} + {K_{MD}}}}} \right){X_2} - {\gamma _2}\left( {1 - \dfrac{{{X_8}}}{{{X_8} + {K_{MD}}}}} \right){X_3} - {\mu _D}{X_3}$\\
	$\dfrac{{d{X_4}}}{{dt}} = {N_{OT}} - {\mu _{PNE}}{X_4}{X_1} - {\mu _{PL}}{X_4}{X_2} - {\mu _N}{X_4}$\\
	$\dfrac{{d{X_5}}}{{dt}} = {O_O} - {\mu _{POE}}{X_5}{X_1} - {\mu _{POL}}{X_5}{X_2} - {\mu _{NO}}{X_5}{X_6} - {\mu _O}{X_5}$\\
	$\dfrac{{d{X_6}}}{{dt}} = {N_{ON}} - {e_{BE}}{X_6}{X_1} - {e_{BL}}{X_6}{X_2} - {e_N}{X_6}$\\
	$\dfrac{{d{X_7}}}{{dt}} = {B_L} + {S_{TL}}\left( {\dfrac{{{X_6}}}{{{X_6} + {K_{MN}}}}} \right)\left( {1 - \dfrac{{{X_5} + {X_4}}}{{{X_5} + {X_4} + {K_{MO}} + {K_{MT}}}}} \right){X_7} - {\mu _{LXP}}{X_7}$\\
	\[\frac{{d{X_8}}}{{dt}} = {B_D} + {S_{TD}}\left( {\frac{{{X_6}}}{{{X_6} + {K_{MN}}*{R_{pf}}}}} \right)\left( {1 - \frac{{{X_5} + {X_4}}}{{{X_5} + {X_4} + \left( {{K_{MO}} + {K_{MT}}} \right)*{R_{pf}}}}} \right){X_8}\]
	\hspace{1.1cm}$ - {\mu _{DXP}}{X_8}$
\end{example}

As easily observed, all kinetic functions belong to $\text{PQK}(\mathscr{N})$. 
In view of the simple forms of the denominators in this case, we can also work with an ``lesser common denominator'' analogous to that for HTK instead of the full product.  
Since there are six distinct denominators, four of which are binomials and two trinomials, the associated poly-PL system will have 144 terms in its canonical PL-representation. We are currently developing computational tools to analyze multiplicity and concentration robustness in such systems.

\subsection{Groups generated by multiplicative monoids of kinetics}
A further observation is that the above construction applies to any multiplicative monoid (semigroup with identity) of kinetics and the group it generates. In particular, ACR and BCR properties are identical.

Our PYK approach has also shown the usefulness of transformations, i.e., dynamic equivalences which leave the stoichiometric subspace invariant. This suggests that one should go beyond the standard implication ``dynamic equivalence $\Rightarrow$ linear conjugacy'' to the longer chain of relations ``transformation $\Rightarrow$ dynamic equivalence $\Rightarrow$ linear conjugacy $\Rightarrow$ equilibria equality'', and still obtain interesting results.

\section{Summary, Conclusions, and Outlook}
\label{sec:sum}
We summarize our results and provide some direction for future research.
\begin{itemize}
	\item[1.] We associate a unique positive linear combination of power-law (PL) kinetic system to a given Hill-type kinetic (HTK) system, and made a comparison with a case specific procedure for computing positive equilibria. Note that the equilibria sets of the two systems coincide. This leads us to identify two important subsets of Hill-type kinetics: the PL-equilibrated and the PL-complex balanced kinetics.
	\item[2.] We used recent results on absolute concentration robustness (ACR) in a species $X$ to establish the Shinar-Feinberg ACR Theorem for PL-equilibrated HT-RDK systems (the subset of complex factorizable HTK systems). This serves as a foundation for the study of ACR in HTK systems.
	\item[3.] Larger and higher deficiency HTK systems with ACR in a species $X$ were explored by the use of the Feinberg's Theorem for independent network decompositions.
	\item[4.] We extended some results of M\"uller and Regensburger on generalized mass action kinetic systems to PL-complex balanced HT-RDK systems.
	\item[5.] We derived the theory of balanced concentration robustness (BCR) in an analogous manner to ACR for PL-equilibrated systems.
	\item[6.] We discussed the extension of our results to more general kinetics including poly-PL quotients (i.e., both the numerator and denominator are nonnegative linear combination of power-law functions).
	\item[7.] As future perspective, it is worth working on computational approaches or tools for analysis of ACR and BCR using our results in this paper.
\end{itemize}
%\begin{itemize}
%\item[1.] We 
%\end{itemize}

%\begin{acknowledgements}
%\end{acknowledgements}



\appendix
\section{Nomenclature}
\subsection{List of abbreviations}
%\FloatBarrier
%\begin{table}
%\caption{caption}
%\label{tab:ap1}       % Give a unique label
%\centering
\begin{tabular}{ll}
	%\hline\noalign{\smallskip}
	%first & second \\
	\noalign{\smallskip}\hline\noalign{\smallskip}
	Abbreviation& Meaning \\
	\noalign{\smallskip}\hline\noalign{\smallskip}
	ACR& absolute concentration robustness\\
	BCR& balanced concentration robustness\\
	CFK& complex factorizable kinetic(s)\\
	CFRF& complex formation rate function\\
	CRN& chemical reaction network\\
	GMAS& generalized mass action system\\
	HTK& Hill-type kinetic(s)\\
	MAK& mass action kinetic(s)\\
	MSA& multistationarity algorithm\\
	NFK& non-complex factorizable kinetic(s)\\
	PLK& power-law kinetic(s)\\
	PYK& poly-PL kinetic(s)\\
	RIDK& rate contant-interaction map decomposable kinetic(s)\\
	SFRF& species formation rate function\\
	\noalign{\smallskip}\hline
\end{tabular}
%\end{table}
%\FloatBarrier
\subsection{List of important symbols}
%\caption{caption}
%\label{tab:ap2}       % Give a unique label
%\centering
\begin{tabular}{ll}
	%\hline\noalign{\smallskip}
	%first & second \\
	\noalign{\smallskip}\hline\noalign{\smallskip}
	Meaning& Symbol \\
	\noalign{\smallskip}\hline\noalign{\smallskip}
	deficiency& $\delta$  \\
	dimension of the stoichiometric subspace& $s$   \\
	incidence matrix& $I_a$\\
	kinetic deficiency& $\widetilde\delta$  \\
	kinetic reactant deficiency& $\widehat\delta$  \\
	molecularity matrix& $Y$\\
	stoichiometric matrix& $N$\\
	stoichiometric subspace& $S$\\
	\noalign{\smallskip}\hline
\end{tabular}

\end{document}